\documentclass[final]{siamltex}

\usepackage{graphics,graphicx}
\usepackage{upquote,verbatim}
\newlength\myverbindent 
\makeatletter
\def\verbatim@processline{%
 \hspace{\myverbindent}\the\verbatim@line\par}
\makeatother
\def\C{{\bf C}}
\def\res{\hbox{\rm res\kern .7pt}}
\def\Oe{O_{\varepsilon}}
\def\diag{\hbox{\rm diag\kern .7pt}}
\def\Zm{Z^{(m)}}
\def\Fm{F^{(m)}}
\def\Am{A^{(m)}}
\def\pput(#1,#2)#3{\noindent\smash{\raise#2pt\hbox to 0pt
   {\kern #1pt #3\hss}}\ignorespaces}

\title{The AAA algorithm for rational approximation}

\author{Yuji Nakatsukasa\thanks{\texttt{nakatsukasa@maths.ox.ac.uk},
Mathematical Institute, University of Oxford, Oxford, OX2 6GG, UK.}\and
Olivier S\`ete\thanks{\texttt{sete@math.tu-berlin.de}, Institute of
Mathematics, TU Berlin, Stra{\ss}e des 17.\ Juni 136,
10623 Berlin, Germany.} \and Lloyd
N.~Trefethen\thanks{\texttt{trefethen@maths.ox.ac.uk}, Mathematical
Institute, University of Oxford, Oxford, OX2 6GG, UK.} \ YN
was supported by the Japan Society for the Promotion
of Science as an Overseas Research Fellow.\ \ OS and LNT
were supported by the European Research Council under the European
Union's Seventh Framework Programme (FP7/2007--2013)/ERC grant
agreement 291068.\ \ The views expressed in this article are
not those of the ERC or the European Commission, and the European
Union is not liable for any use that may be made of the information
contained here.}

\begin{document}

\maketitle


\begin{quote}
\em For Jean-Paul Berrut, the pioneer of numerical algorithms based on
rational barycentric representations, on his 65th birthday.
\end{quote}

\bigskip

\begin{abstract}
We introduce a new algorithm for approximation by rational
functions on a real or complex set of points,
implementable in 40 lines of Matlab and requiring no user
input parameters.  Even on a disk or interval
the algorithm may outperform existing methods, and on more
complicated domains it is especially competitive.  The core
ideas are (1) representation of the rational approximant in
barycentric form with interpolation at certain support points and
(2) greedy selection of the support points to avoid exponential
instabilities.  The name AAA stands for
``adaptive Antoulas--Anderson'' in honor of the authors who
introduced a scheme based on (1).  We present the core algorithm
with a Matlab code and nine applications and describe variants
targeted at problems of different kinds. Comparisons are made
with vector fitting, RKFIT, and other existing methods for
rational approximation.
\end{abstract}


\pagestyle{myheadings}
\thispagestyle{plain}
\markboth{\sc Nakatsukasa, S\`ete, and Trefethen}
{\sc AAA algorithm for rational approximation}

\section{Introduction}
Rational approximations of real or complex functions are used
mainly in two kinds of applications.  Sometimes they provide
compact representations of functions, much more efficient than
polynomials for functions with poles or other singularities
on or near the domain of approximation or on unbounded domains.
Other times, their role is one of extrapolation: the extraction of
information about poles or values or other properties of a function
in regions of the real line or complex plane beyond where it is
known a priori.  For example, standard methods of acceleration of
convergence of sequences and series, such as the eta and epsilon
algorithms, are based on rational approximations~\cite{bgm,bz}.
For a general discussion of the uses of rational approximation,
see Chapter~23 of~\cite{atap}, and for theoretical foundations,
see~\cite{braess}.

Working with rational approximations, however, can be problematic.
There are various challenges here, one of which particularly
grabs attention: {\em spurious poles,} also known as {\em
Froissart doublets,} which can be regarded either as poles with
very small residues or as pole-zero pairs so close together as to
nearly cancel~\cite{bz13,froissart,gk03,gp97,gp99,stahlspurious}.
Froissart doublets arise in the fundamental mathematical problem
--- i.e., in ``exact arithmetic'' --- and are the reason why
theorems on convergence of rational approximations, e.g.\ of
Pad\'e approximants along diagonals of the Pad\'e table, typically
cannot hold without the qualification of convergence in capacity
rather than uniform convergence~\cite{bgm,pomm}.  On a computer
in floating point arithmetic, they arise all the more often;
we speak of {\em numerical Froissart doublets,} recognizable by
residues on the order of machine precision.  These difficulties
are related to the fact that the problem of analytic continuation,
for which rational approximation is the most powerful general
technique, is ill-posed.  See Chapter~26 of~\cite{atap}.

In this paper we propose a {\em AAA algorithm\/} for rational
approximation that offers a speed, flexibility, and robustness
we have not seen in other algorithms; the name stands for
``adaptive Antoulas--Anderson''\footnote{We write ``a AAA''
rather than ``an AAA'' because in speech we say ``triple-A.''}.
(More recent material related to the Antoulas--Anderson method can
be found in~\cite{ionita}, and a
discussion of related methods is given in Section 11.)
The algorithm combines two ideas.  First, following Antoulas and
Anderson~\cite{aa} (though their presentation of the mathematics
is very different), rational functions are represented in
barycentric form with interpolation at certain support points
selected from a set provided by the user. 
Second, the algorithm
grows the approximation degree one by one, selecting support
points in a systematic greedy fashion so as to avoid exponential
instabilities.  Numerical Froissart doublets usually
do not appear, and if they do, they can usually be removed by
one further solution of a least-squares problem.

Perhaps the most striking feature of the AAA algorithm is that
it is not tied to a particular domain of approximation such as
an interval, a circle, a disk, or a point.  Many methods for
rational approximation utilize bases that are domain-dependent,
whereas the AAA barycentric representation, combined with its
adaptive selection of support points, avoids such a dependence.
The algorithm works effectively with point sets that may include
discretizations of disconnected regions of irregular shape,
possibly unbounded, and the functions approximated may have poles
lying in the midst of the sample points.  Thus the AAA algorithm
is fast and flexible, but on the other hand, it does not claim
to achieve optimality in any particular norm such as $L^2$ or
$L^\infty$.  For such problems more specialized methods may be
used, though as we shall mention in Section~\ref{sec-variants} and
have subsequently developed in~\cite{fntb}, the AAA algorithm may
still play a role in providing an initial guess or as the design pattern
for a variant algorithm based, for example, on iterative reweighting.

\section{\label{algebra}Rational barycentric representations}
The barycentric formula takes the form
of a quotient of two partial fractions,
\begin{equation}
r(z) = {n(z)\over d(z)} = {\sum_{j=1}^m \left. {w_j^{} f_j^{}\over z- z_j^{}} \right/
\sum_{j=1}^m  {w_j^{} \over z- z_j^{}}},
\label{bary}
\end{equation}
where $m\ge 1$ is an integer, $z_1^{},\dots,z_m^{}$ are
a set of real or complex distinct {\em support points,}
$f_1^{},\dots,f_m^{}$ are a set of real or complex {\em data
values,} and $w_1^{},\dots,w_m^{}$ are a set of real or complex
{\em weights.}  As indicated in the equation, we let $n(z)$
and $d(z)$ stand for the partial fractions in the numerator and
the denominator.  When we wish to be explicit about step numbers,
we write $r_m^{}(z)= n_m^{}(z)/d_m^{}(z)$.

The {\em node polynomial\/} $\ell\kern .5pt$ associated with the set
$z_1^{},\dots,z_m^{}$ is the monic polynomial of degree $m$ with
these numbers as roots,
\begin{equation}
\ell(z) = \prod_{j=1}^m (z-z_j^{}).
\label{node}
\end{equation}
If we define
\begin{equation}
p(z) = \ell(z) n(z), \quad q(z) = \ell(z) d(z),
\end{equation}
then $p$ and $q$ are each polynomials of degree at most $m-1$.
Thus we have a simple link between the barycentric
representation of $r$ and its more familiar representation as
a quotient of polynomials,
\begin{equation}
r(z) = {p(z)/\ell(z)\over q(z)/\ell(z)}
= {p(z)\over q(z)}.
\label{link}
\end{equation}
This equation tells us that $r$ is a rational function of type
$(m-1,m-1)$, where, following standard terminology, we say that
a rational function is of {\em type $(\kern .5pt\mu,\nu)$} for
integers $\mu,\nu \ge 0$ if it can be written as a quotient of a
polynomial of degree at most~$\mu$ and a polynomial of degree at
most $\nu\kern .7pt $, not necessarily in lowest terms.\footnote{As
a special case it is also customary to say that the zero function
is not only of type $(\kern .5pt \mu,\nu)$ for any $\mu,\nu\ge 0$,
but also of type $(-\infty,\nu)$ for any $\nu\ge 0$.} The numerator
$n(z)$ and denominator $d(z)$ are each of type $(m-1,m)$, so it
is obvious from (\ref{bary}) that~$r$ is of type $(2m-1,2m-1)$;
it is the cancellation of the factors $1/\ell(z)$ top and bottom
that makes it actually of type $(m-1,m-1)$.  This is the paradox
of the barycentric formula: it is of a smaller type than it looks,
and its poles can be anywhere except where they appear to be (assuming
the weights $w_j^{}$ in~(\ref{bary}) are nonzero).
The paradox goes further in that, given any set of support points
$\{z_j^{}\}$, there is a special choice of the weights $\{w_j^{}\}$
for which~$r$ becomes a polynomial of degree $m-1$.  This is
important in numerical computation, providing a numerically stable
method for polynomial interpolation even in thousands of points;
see~\cite{bt,chebfun,atap} for discussion and references.  However,
it is not our subject here.  Here we are concerned with the cases
where (\ref{bary}) is truly a rational function, a situation
exploited perhaps first by Salzer~\cite{salzer} and Schneider and
Werner~\cite{sw} and most importantly in subsequent years
by Berrut and his collaborators~\cite{berrut,berrut00,bbm,bt,fh,klein}.
For further links to the literature see Section~11.

A key aspect of (\ref{bary}) is its interpolatory property.  At
each point $z_j^{}$ with $w_j^{}\ne 0$, the formula is undefined,
taking the form $\infty / \infty$ (assuming $f_j \ne 0$).  However,
this is a removable singularity, for $\lim_{z\to z_j^{}} r(z)$
exists and is equal to $f_j^{}$.  Thus if the weights~$w_j^{}$
are nonzero, (\ref{bary}) provides a type $(m-1,m-1)$ {\em
rational interpolant} to the data $f_1^{},\dots,f_m^{}$ at
$z_1^{},\dots,z_m^{}$.  Note that such a function has $2m-1$
degrees of freedom, so roughly half of these are fixed by the
interpolation conditions and the other half are not.

We summarize the properties of barycentric representations
developed in the discussion above by the following theorem.

\medskip

\begin{theorem}[Rational barycentric representations]
Let $z_1^{},\dots, z_m^{}$ be an arbitrary set of distinct complex
numbers.  As $f_1^{},\dots ,f_m^{}$ range over all complex values
and $w_1^{},\dots ,w_m^{}$ range over all nonzero
complex values, the functions
\begin{equation}
r(z) = {n(z)\over d(z)} =
{\sum_{j=1}^m \left. {w_j^{} f_j^{}\over z- z_j^{}} \right/
\sum_{j=1}^m  {w_j^{} \over z- z_j^{}}}
\label{baryagain}
\end{equation}
range over the set of all
rational functions of type $(m-1,m-1)$ that have
no poles at the points $z_j^{}$.
Moreover, $r(z_j^{}) = f_j^{}$ for each $j$.
\end{theorem}

\begin{proof}
By (\ref{link}), any quotient $n/d$ as in (\ref{baryagain}) is a rational
function $r$ of type $(m-1,m-1)$.  
Moreover, since $w_j\ne 0$, $d$ has a simple pole at $z_j$ and
$n$ has either a simple pole there (if $f_j^{}\ne 0$) or no pole.  Therefore
$r$ has no pole at $z_j^{}$.

Conversely, suppose $r$ is a rational
function of type $(m-1,m-1)$ with no poles at
the points $z_j^{}$, and write $r = p/q$ where $p$ and $q$ are
polynomials of degree at most $m-1$ with no common zeros. 
Then $q/\ell$ is a rational function with a zero at $\infty$
and simple poles at the points $z_j^{}$.  Therefore $q/\ell$
can be written in the partial fraction form of a denominator
$d$ as in (\ref{baryagain}) with $w_j \ne 0 $ for each $j$ (see
Theorem 4.4h and p.~553 of~\cite{henrici}).
Similarly $p/\ell$ is a rational function with a zero at $\infty$
and simple poles at the points $z_j^{}$ or a subset of them.
Therefore, since $w_j\ne 0$, $p/\ell$
can be written in the partial fraction form of a numerator
$n$ as in (\ref{baryagain}).
\end{proof}

\smallskip

If the support points $\{z_j^{}\}$ have no influence on the set
of functions described by (\ref{baryagain}), one may wonder, what
is the use of barycentric representations?  The answer is all
about numerical quality of the representation and is at the very
heart of why the AAA algorithm is so effective.  The barycentric
formula is composed from quotients $1/(z-z_j^{})$, and for good
choices of $\{z_j^{}\}$, these functions are independent enough
to make the representation well-conditioned --- often far
better conditioned, in particular, than one would find with a
representation $p(z)/q(z)$.  As shown in
section~\ref{comparisons}, they are also better conditioned
that the partial fraction representations used by vector
fitting, since in that case, the points $\{z_j^{}\}$ are 
constrained to be the poles of $r$.

The use of localized and sometimes
singular basis functions is an established theme in other areas of
scientific computing.  {\em Radial basis functions,} for example,
have excellent conditioning properties when they are composed
of pieces that are well separated~\cite{ff}.  In~\cite{dfq}
one even finds a barycentric-style quotient of two RBF sums utilized
with good effect.  Similarly the {\em
method of fundamental solutions,} which has had great success in
solving elliptic PDEs such as Helmholtz problems, represents its
functions as linear combinations of Hankel or other functions
each localized at a singular point~\cite{bb}.  An aim of the
present paper is to bring this kind of thinking to the subject
of function theory.  Yet another related technique in scientific
computing is {\em discretizations of the Cauchy integral formula,}
for example by the trapezoidal rule on the unit circle, to evaluate
analytic functions inside a curve.  The basis functions implicit
in such a discretization are singular, introducing poles and hence
errors of size $\infty$ at precisely the data points where one
might expect the errors to be $0$, but still the approximation
may be excellent away from the curve~\cite{akt,qbx}.

\section{\label{core}Core AAA algorithm}
We begin with a finite {\em sample set\/} $Z \subseteq \C$ of
$M\gg 1$ points.  We assume a function $f(z)$ is given that is
defined at least for all $z\in Z$.  This function may have an analytic
expression, or it may be just a set of data values.

The AAA algorithm takes the form of an iteration for $m = 1, 2, 3,
\dots ,$ with $r$ represented at each step in the barycentric form
(\ref{baryagain}).  At step $m$ we first pick the next support
point $z_m^{}$ by the greedy algorithm to be described below,
and then we compute corresponding weights $w_1^{},\dots, w_m^{}$
by solving a linear least-squares problem over the subset of
sample points that have not been selected as support points,
\begin{equation} \Zm = Z \backslash \{z_1^{},\dots,z_m{}\}.
\end{equation} Thus at step $m$, we compute a rational function $r$
of type $(m-1,m-1)$, which generically will interpolate $f_1^{} =
f(z_1^{}),\dots,f_m^{}= f(z_m^{})$ at $z_1^{},\dots,z_m^{}$ (though
not always, since one or more weights may turn out to be zero).

The least-squares aspect of the algorithm is as follows.
Our aim is an approximation
\begin{equation}
f(z) \approx {n(z) \over d(z)}, \quad z\in Z,
\label{oldfit}
\end{equation}
which in linearized form becomes
\begin{equation}
f(z)d(z) \approx n(z), \quad z \in \Zm.
\label{newfit}
\end{equation}
Note that in going from
(\ref{oldfit}) to (\ref{newfit}) we have replaced
$Z$ by $\Zm$, because $n(z)$ and
$d(z)$ will generically have poles at $z_1^{},\dots,z_m^{}$,
so (\ref{newfit}) would not make sense over all $z\in Z$.
The weights $w_1^{},\dots, w_m^{}$ are chosen to solve
the least-squares problem
\begin{equation}
\hbox{minimize\kern 1pt } \| fd - n\|_{\Zm}^{}, \quad \|w\|_m^{}=1,
\label{l2prob}
\end{equation}
where $\|\cdot\|_{\Zm}^{}$ is the discrete 2-norm over $\Zm$ and
$\|\cdot\|_m^{}$ is the discrete 2-norm on $m$-vectors.
To ensure that this problem makes sense, we assume that $\Zm$ has at least
$m$ points, i.e., $m \le M/2$.

The greedy aspect of the iteration is as follows.  At step $m$, the
next support point $z_m^{}$ is chosen as a point $z\in Z^{(m-1)}$
where the nonlinear residual $f(z)-n(z)/d(z)$ at step $m-1$
takes its maximum absolute value.

Assuming the iteration is successful, it terminates when the
nonlinear residual is sufficiently small; we have found it
effective to use a default tolerance of $10^{-13}$ relative to the
maximum of $|f(Z)|$.  The resulting approximation typically has
few or no numerical Froissart doublets, and if there are any,
they can usually be removed by one further least-squares step
to be described in Section~\ref{removing}.  (If the convergence
tolerance is too tight, the approximation will stagnate and many
Froissart doublets will appear.)  In the core AAA algorithm it
is an approximation of type $(m-1,m-1)$.

It remains to spell out the linear algebra involved in (\ref{l2prob}).
Let us regard $\Zm$ and $\Fm = f(\Zm)$ as column vectors,
\begin{displaymath}
\Zm = (Z_1^{(m)},\dots, Z_{M-m}^{(m)})^T, \quad
\Fm = (F_1^{(m)},\dots, F_{M-m}^{(m)})^T.
\end{displaymath}
We seek a normalized column vector
\begin{displaymath}
w = (w_1^{},\dots, w_m^{})^T, \quad \|w\|_m = 1
\end{displaymath}
that minimizes the 2-norm of the $(M-m)$-vector
\begin{displaymath}
\sum_{j=1}^m {w_j^{} F_i^{(m)}\over Z_i^{(m)}-z_j} -
\sum_{j=1}^m {w_j^{} f_j^{}\over Z_i^{(m)}-z_j},
\end{displaymath}
that is,
\begin{displaymath}
\sum_{j=1}^m {w_j^{} (F_i^{(m)}-f_j^{})\over Z_i^{(m)}-z_j}.
\end{displaymath}
This is a matrix problem of the form
\begin{equation}
\hbox{minimize \kern 1pt} \|\Am w\|_{M-m}^{}, \quad \|w\|_m^{}=1,
\label{minprob}
\end{equation}
where $\Am $ is the $(M-m)\times m$ {\em Loewner matrix}~\cite{antoulas}
\begin{equation}
\Am  = 
\pmatrix{\displaystyle{F_1^{(m)}-f_1^{}\over Z_1^{(m)}-z_1^{}} &
\cdots & \displaystyle{F_1^{(m)}-f_m^{}\over Z_1^{(m)}-z_m^{}} \cr
\noalign{\vskip .3in}
\vdots &\ddots & \vdots \cr
\noalign{\vskip .3in}
\displaystyle{F_{M-m}^{(m)}-f_1^{}\over Z_{M-m}^{(m)}-z_1^{}}
& \cdots & \displaystyle{F_{M-m}^{(m)}-f_m^{}\over Z_{M-m}^{(m)}-z_m^{}} }.
\label{nla}
\end{equation}
We will solve (\ref{minprob}) using the singular value decomposition (SVD),
taking $w$ as the final right singular vector in a reduced
SVD $\Am = U\Sigma V^*$.
(The minimal singular value of $\Am$ might be nonunique or
nearly so, but our algorithm does not rely on its uniqueness.)
Along the way it is convenient to make use of
the $(M-m)\times m$ {\em Cauchy matrix}
\begin{equation}
C = 
\pmatrix{\displaystyle{1\over Z_1^{(m)}-z_1^{}} &
\cdots & \displaystyle{1\over Z_1^{(m)}-z_m^{}} \cr
\noalign{\vskip .3in}
\vdots & \ddots & \vdots \cr
\noalign{\vskip .3in}
\displaystyle{1\over Z_{M-m}^{(m)}-z_1^{}}
& \cdots & \displaystyle{1\over Z_{M-m}^{(m)}-z_m^{}}},
\label{Cmatrix}
\end{equation}
whose columns define the basis in which we approximate.
If we define diagonal left and right scaling matrices by
\begin{equation}
S_F^{} = \diag(F_1^{(m)},\dots,F_{M-m}^{(m)}), \quad
S_f^{} = \diag(f_1^{},\dots,f_m^{}),
\end{equation}
then we can construct $\Am $ from $C$ using the identity
\begin{equation}
\Am  = S_F^{}C - CS_f^{},
\label{AfromC}
\end{equation}
and once $w$ is found with the SVD, we can compute 
$(M-m)$-vectors $N$ and $D$ with
\begin{equation}
N = C(wf),  \quad D = Cw.
\end{equation}
These correspond to the values of $n(z)$ and $d(z)$ at points
$z\in \Zm$.  (Since $M$ will generally be large, it is important
that the sparsity of $S_F^{}$ is exploited in the multiplication of
(\ref{AfromC}).)  Finally, to get an $M$-vector $R$ corresponding
to $r(z)$ for all $z\in Z$, we set $R = f(Z)$ and then $R(\Zm) = N/D$.

After the AAA algorithm terminates (assuming $w_j^{}\ne 0$ for all $j$),
one has a rational approximation
$r(z) = n(z)/d(z)$ in barycentric form.  The zeros of $d$, which are
(generically) the poles of~$r$, can be computed by solving
an $(m+1)\times (m+1)$ generalized eigenvalue
problem in arrowhead form~\cite[Sec.~2.3.3]{klein},
\begin{equation}
\def\crr{\cr\noalign{\vskip 2pt}}
\pmatrix{0& w_1^{} & w_2^{} & \cdots & w_m^{} \crr
1 & z_1^{} \crr
1 & & z_2^{} \crr
\vdots & & & \ddots \crr
1 & & & & z_m^{}}
= \lambda
\pmatrix{0\crr
&1 \crr
& & 1 \crr
& & & \ddots \crr
& & & & 1 } .
\end{equation}
At least two of the eigenvalues of this problem are infinite, and
the remaining $m-1$ are the zeros of $d$.  A similar computation
with $w_j^{}$ replaced by $w_j^{}f_j^{}$ gives the zeros of $n(z)$.

The following proposition collects some elementary properties of
the core AAA algorithm.  We say ``a'' instead of ``the'' in view
of the fact that in cases of ties in the greedy choice at each
step, AAA approximants are not unique.

\begin{proposition}
\label{prop}
Let $r(z)$ be a AAA approximant at step\/ $m$ of a function
$f(z)$ on a set $Z$ (computed in exact arithmetic).  The following
statements refer to AAA approximants at step $m$, and $a$ and $b$
are complex constants.

{\em Affineness in $f$.}  
For any $a\ne 0$ and $b$,
$ar(z) + b$ is an approximant of $af(z) + b$ on $Z$.

{\em Affineness in $z$.}  
For any $a\ne 0$ and $b$,
$r(az+b)$ is an approximant of $f(az+b)$ on $(Z-b)/a$.

{\em Monotonicity.}  
The linearized residual norm
$\sigma_{\min}(\Am  )= \|fd-n\|_{Z^{(m)}}^{}$ is a nonincreasing
function of~$m$.
\end{proposition}

\begin{proof}
These properties are straightforward and we do not spell out
the arguments except to note that the monotonicity property
follows from the fact that $\Am $ is obtained from $A^{(m-1)}$
by deleting one row and appending one column.  Since the minimum
singular vector for $A^{(m-1)}$ (padded with one more zero)
is also a candidate singular vector of $\Am $, we must have
$\sigma_{\min{}}(\Am ) \le\sigma_{\min{}}(A^{(m-1)})$.
\end{proof}

One might ask, must the monotonicity be strict, with
$\sigma_{\min{}}(\Am ) <\sigma_{\min{}}(A^{(m-1)})$
if $\sigma_{\min{}}(A^{(m-1)})\ne 0\kern .7pt$?
So far as we are aware, the answer is no.  An equality
$\sigma_{\min{}}(\Am ) = \sigma_{\min{}}(A^{(m-1)})$ implies
$f(z_m^{})d_{m-1}^{}(z_m^{}) =n_{m-1}^{}(z_m^{})$, where
$z_m^{}$ is the support point selected at step $m$.  This in
turn implies $d_{m-1}^{}(z_m^{}) =n_{m-1}^{}(z_m^{}) = 0$,
since otherwise we could divide by $d_{m-1}^{}(z_m^{})$ to find
$f(z_m^{})-n_{m-1}^{}(z_m^{})/d_{m-1}^{}(z_m^{}) = 0$, which
would contradict the greedy choice of $z_m^{}$.  But so far as
we know, the possibility $d_{m-1}^{}(z_m^{}) =n_{m-1}^{}(z_m^{})
= 0$ is not excluded.

Since the AAA algorithm involves SVDs of dimensions $(M-j)\times
j$ with $j = 1,2,\dots ,m,$ its complexity is $O(Mm^3)$ flops.
This is usually modest since in most applications $m$ is small.

\section{Matlab code}

\begin{figure}
{\footnotesize
\begin{verbatim}
function [r,pol,res,zer,z,f,w,errvec] = aaa(F,Z,tol,mmax)
% aaa  rational approximation of data F on set Z
%      [r,pol,res,zer,z,f,w,errvec] = aaa(F,Z,tol,mmax)
%
%  Input: F = vector of data values, or a function handle
%         Z = vector of sample points
%         tol = relative tolerance tol, set to 1e-13 if omitted
%         mmax: max type is (mmax-1,mmax-1), set to 100 if omitted
%
% Output: r = AAA approximant to F (function handle)
%         pol,res,zer = vectors of poles, residues, zeros
%         z,f,w = vectors of support pts, function values, weights
%         errvec = vector of errors at each step

M = length(Z);                            % number of sample points
if nargin<3, tol = 1e-13; end             % default relative tol 1e-13
if nargin<4, mmax = 100; end              % default max type (99,99)
if ~isfloat(F), F = F(Z); end             % convert function handle to vector
Z = Z(:); F = F(:);                       % work with column vectors
SF = spdiags(F,0,M,M);                    % left scaling matrix
J = 1:M; z = []; f = []; C = [];          % initializations
errvec = []; R = mean(F); 
for m = 1:mmax                            % main loop
  [~,j] = max(abs(F-R));                  % select next support point
  z = [z; Z(j)]; f = [f; F(j)];           % update support points, data values
  J(J==j) = [];                           % update index vector
  C = [C 1./(Z-Z(j))];                    % next column of Cauchy matrix
  Sf = diag(f);                           % right scaling matrix
  A = SF*C - C*Sf;                        % Loewner matrix
  [~,~,V] = svd(A(J,:),0);                % SVD
  w = V(:,m);                             % weight vector = min sing vector
  N = C*(w.*f); D = C*w;                  % numerator and denominator
  R = F; R(J) = N(J)./D(J);               % rational approximation
  err = norm(F-R,inf);
  errvec = [errvec; err];                 % max error at sample points
  if err <= tol*norm(F,inf), break, end   % stop if converged
end
r = @(zz) feval(@rhandle,zz,z,f,w);       % AAA approximant as function handle
[pol,res,zer] = prz(r,z,f,w);             % poles, residues, and zeros
[r,pol,res,zer,z,f,w] = ...
   cleanup(r,pol,res,zer,z,f,w,Z,F);      % remove Frois. doublets (optional)

function [pol,res,zer] = prz(r,z,f,w)     % compute poles, residues, zeros
m = length(w); B = eye(m+1); B(1,1) = 0;             
E = [0 w.'; ones(m,1) diag(z)]; 
pol = eig(E,B); pol = pol(~isinf(pol));   % poles
dz = 1e-5*exp(2i*pi*(1:4)/4);
res = r(bsxfun(@plus,pol,dz))*dz.'/4;     % residues
E = [0 (w.*f).'; ones(m,1) diag(z)]; 
zer = eig(E,B); zer = zer(~isinf(zer));   % zeros

function r = rhandle(zz,z,f,w)            % evaluate r at zz
zv = zz(:);                               % vectorize zz if necessary
CC = 1./bsxfun(@minus,zv,z.');            % Cauchy matrix 
r = (CC*(w.*f))./(CC*w);                  % AAA approx as vector
ii = find(isnan(r));                      % find values NaN = Inf/Inf if any
for j = 1:length(ii)
  r(ii(j)) = f(find(zv(ii(j))==z));       % force interpolation there
end
r = reshape(r,size(zz));                  % AAA approx
\end{verbatim}
\par}
\medskip
\caption{\label{codefig} Matlab code for the AAA algorithm,
returning\/ $r$ as a function handle.  The optional code {\tt cleanup} is
listed in Section~$\ref{removing}$.  Commenting out this line gives
the ``core'' AAA algorithm.}
\end{figure}

The Matlab code {\tt aaa.m}, shown in Figure~\ref{codefig} and
available in Chebfun~\cite{chebfun}, is intended to be readable
as well as computationally useful and makes the AAA algorithm
fully precise.  The code closely follows the algorithm description
above, differing in just one detail for programming convenience:
instead of working with a subset $\Zm$ of size $M-m$ of $Z$, we
use an index vector $J$ that is a subset of size $M-m$ of $1{:}M$.
The matrices $C$ and $A$ always have $M$ rows, and the SVD is
applied not to all of $A$ but to the submatrix with rows indexed
by $J$.

We have modeled {\tt aaa} on the code {\tt ratdisk} of~\cite{gpt},
and in particular, the object~{\tt r} that it returns is a
function handle that can evaluate $r$ at a scalar, vector, or
matrix of real or complex numbers.  The code is flexible enough
to cover a wide range of computations, and was used for the
applications of Section~\ref{sec-applics}.  The code as listed does
not include the safeguards needed in a fully developed
piece of software, but its Chebfun realization (version 5.6.0,
December 2016) does have some of these features.

\begin{figure}
\begin{center}
\vskip .2in
\includegraphics[scale=.55]{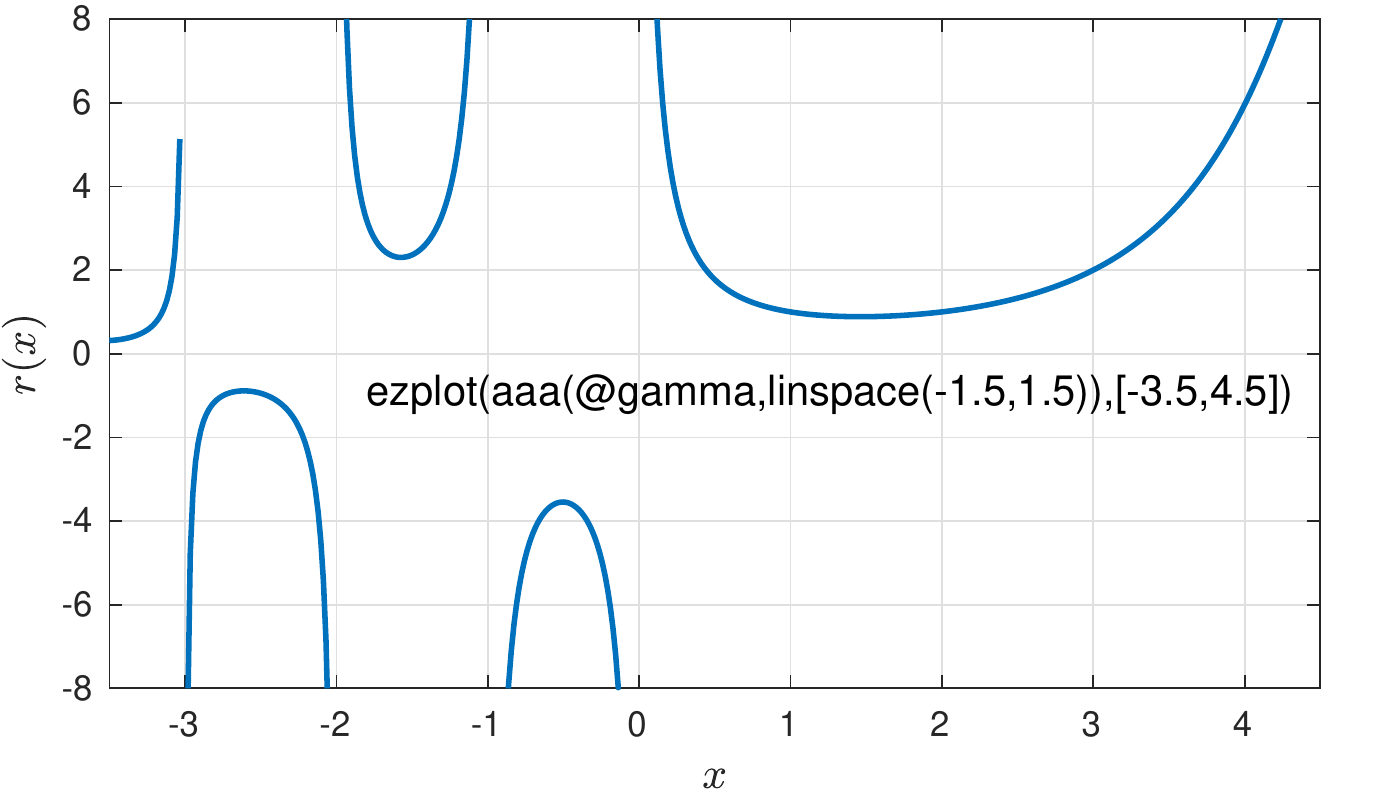}
\vskip -.1in
\end{center}
\caption{\label{ezfig} AAA approximation to
$f(x) = \Gamma(x)$ produced by the command shown.
This is a rational approximant of type $(9,9)$ whose
first four poles match $0$, $-1$, $-2$, and $-3$ to
$15$, $15$, $7$, and $3$ digits, respectively, with
corresponding residues
matching $1$, $-1$, $1/2$, and $-1/6$ to similar accuracy.}
\end{figure}

Computing a AAA approximation with {\tt aaa} can be as simple as
calling {\tt r = aaa(f,Z)} and then evaluating {\tt r(z)}, where
{\tt z} is a scalar or vector or matrix of real or complex numbers.
For example, Figure~\ref{ezfig} shows a good approximation of
the gamma function on $[-3.5,4.5]$ extrapolated from 100 samples
in $[-1.5,1.5]$, all invoked by the line of Matlab displayed in
the figure.

For a second illustration, shown in Figure~\ref{fig1}, we call
{\tt aaa} with the sequence

{\small
\begin{verbatim}

Z = exp(linspace(-.5,.5+.15i*pi,1000));
F = @(z) tan(pi*z/2);
[r,pol,res,zer] = aaa(F,Z);

\end{verbatim}
\par}

\noindent
The set $Z$ is a spiral of 1000 points winding $7 {1\over 2}$ times
around the origin in the complex plane.  When the code is executed,
it takes $m=12$ steps to convergence with the following errors:

{\small
\begin{verbatim}

2.49e+01, 4.28e+01, 1.71e+01, 8.65e-02, 1.27e-02, 9.91e-04,
5.87e-05, 1.29e-06, 3.57e-08, 6.37e-10, 1.67e-11, 1.30e-13

\end{verbatim}
\par}

\begin{figure}
\begin{center}
\vskip .2in
\includegraphics[scale=.45]{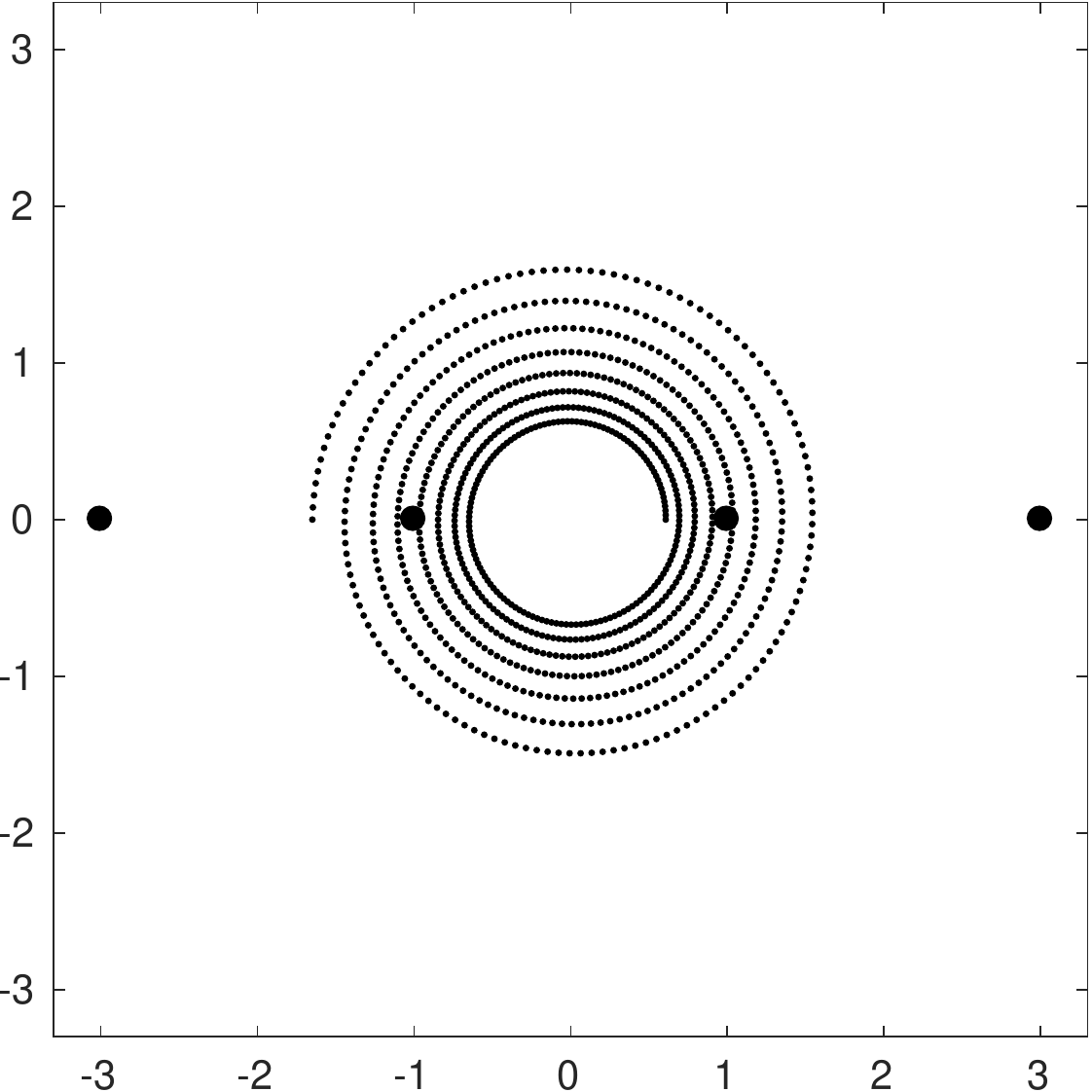}
\end{center}
\caption{\label{fig1} Approximation of $f(z) = \tan(\pi z/2)$ on
a set of\/ $1000$ points lying along a spiral in the complex plane.  
The inner pair of computed poles (dots) 
match the corresponding poles of $f$ to $15$ digits, and the second
pair match to $7$ digits.}
\end{figure}

\noindent  The first pair
of poles of $r$ match the poles of $f$ at $\pm 1$ to 15 digits, the next
pair match $\pm 3$ to 7 digits, and the third pair match $\pm 5$
to 3 digits.  The zeros of $r$ show a
similar agreement with those of $f$.

{\small
\begin{verbatim}

         Poles                                Zeros
  1.00000000000000 + 9.01e-17i        -0.00000000000000 - 2.14e-15i
 -1.00000000000000 + 5.55e-19i         2.00000000000839 - 4.77e-12i
  3.000000100      - 4.09e-08i        -2.00000000000043 - 4.38e-12i
 -3.000000065      - 5.83e-08i         4.0000427        - 1.25e-05i
  5.002693         - 5.61e-04i        -4.0000355        - 1.73e-05i
 -5.002439         - 7.47e-04i         6.0461           - 6.74e-03i
  7.3273           - 3.46e-02i        -6.0435           - 8.67e-03i
 -7.3154           - 4.29e-02i         9.387            - 1.17e-01i
 13.65             - 3.74e-01i        -9.350            - 1.39e-01i
-13.54             - 3.74e-01i       -26.32             - 1.85e+00i
-47.3              - 3.60e+02i        26.81             - 1.80e+00i

\end{verbatim}
\par}

\section{\label{removing}Removing numerical Froissart doublets}
In many applications, the AAA algorithm finishes with a clean
rational approximation $r$, with no numerical Froissart doublets.
Sometimes, however, these artifacts appear, and one can always
generate them by setting the convergence tolerance to $0$ and
taking the maximal step number {\tt mmax} to be sufficiently large.
Our method of addressing this problem, implemented in the code
{\tt cleanup} listed in Figure~\ref{codefig2}, is as follows.
We identify spurious poles by their residues $< 10^{-13}$, remove
the nearest support points from the set of support points, and
then solve the least-squares problem (\ref{l2prob}) by a new
SVD calculation.

For an example from~\cite{gpt}, suppose the function $f(z) =
\log(2+z^4)/(1-16z^4)$ is approximated in 1000 roots of unity
with tolerance~$0$.  The left image of Figure~\ref{spurious} shows
what happens if this is done with {\tt cleanup} disabled: the code
runs to $m=100$, at which point there are 58 numerical Froissart
doublets near the unit circle, marked in red.  The right image,
with {\tt cleanup} enabled, shows just one Froissart doublet,
an artifact that appeared when the first 58 were eliminated.
The other poles comprise four inside the unit circle, matching
poles of $f$, and collections of poles lining up along branch
cuts of $f$; see Application~\ref{sec-applics}.2.

\begin{figure}
{\footnotesize
\hspace{1in}
\begin{verbatim}
function [r,pol,res,zer,z,f,w] = cleanup(r,pol,res,zer,z,f,w,Z,F)
m = length(z); M = length(Z);
ii = find(abs(res)<1e-13);            % find negligible residues
ni = length(ii);
if ni == 0, return, end
fprintf('%d Froissart doublets\n',ni)
for j = 1:ni
  azp = abs(z-pol(ii(j)));
  jj = find(azp == min(azp),1);
  z(jj) = []; f(jj) = [];             % remove nearest support points
end
for j = 1:length(z)
  F(Z==z(j)) = []; Z(Z==z(j)) = [];
end
m = m-length(ii);
SF = spdiags(F,0,M-m,M-m);
Sf = diag(f);
C = 1./bsxfun(@minus,Z,z.');
A = SF*C - C*Sf;
[~,~,V] = svd(A,0); w = V(:,m);       % solve least-squares problem again
r = @(zz) feval(@rhandle,zz,z,f,w);   
[pol,res,zer] = prz(r,z,f,w);         % poles, residues, and zeros
\end{verbatim}
\par}
\medskip
\caption{\label{codefig2} Matlab code for removing numerical Froissart doublets.}
\end{figure}

\begin{figure}
\begin{center}
\vskip .2in
\includegraphics[scale=.75]{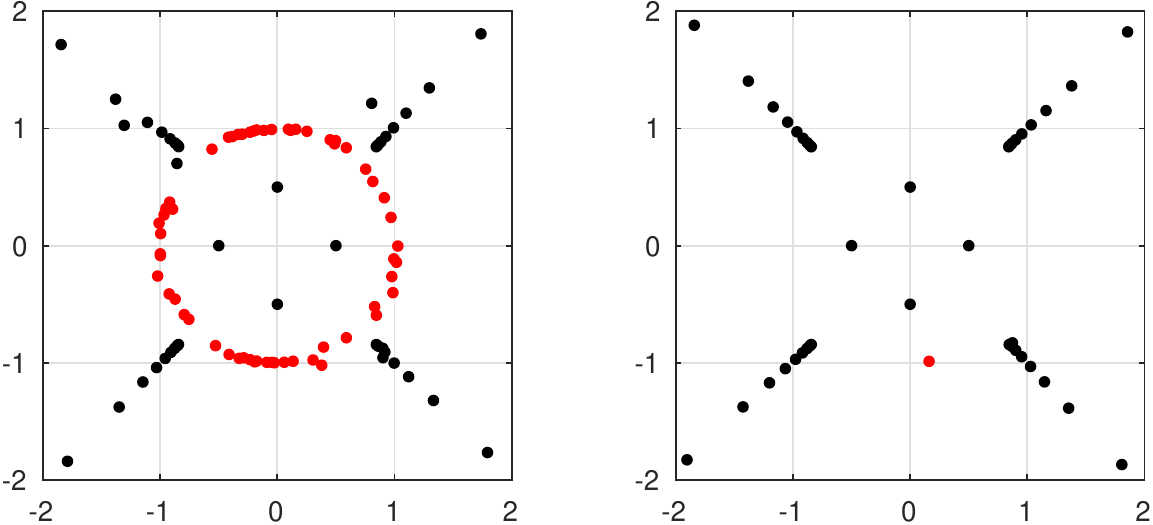}
\end{center}
\caption{\label{spurious} If the AAA algorithm is run with ${\tt tol} =0$ and
with the {\tt cleanup} option disabled, as in the left image,
numerical Froissart doublets appear.  Here we see approximation
of $f(z) = \log(2+z^4)/(1-16z^4)$ in $1000$ roots of unity, with
red dots marking poles with residues of absolute value $<10^{-13}$. 
The right image shows that if {\tt cleanup} is enabled, there is
just one doublet.
If the algorithm is run with its default tolerance $10^{-13}$, then
the result is much the same but with no numerical doublets at all.}
\end{figure}

The introduction cited a number of
papers related to the problem of Froissart
doublets~\cite{bz13,froissart,gk03,gp97,gp99,stahlspurious}.
This literature is mainly concerned with Pad\'e approximation
(rational approximation of a Taylor series at a point) and with
representations $p/q$, but \cite{bz13} discusses Froissart doublets
in the context of barycentric representations.

\section{\label{sec-applics}Applications}
In this section we present nine applications computed with
{\tt aaa}.  Different applications raise different mathematical
and computational issues, which we discuss in some generality
in each case while still focusing on the particular example for
concreteness.  Application~\ref{sec-applics}.6 is the only one
in which any numerical Froissart doublets were found (6 of them);
they are removed by {\tt cleanup}.

\subsection{\label{sec-app1}Analytic functions in the unit disk}
Let $\Delta$ be the closed unit disk $\Delta = \{z\in {\bf
C}: \, |z|\le 1\}$, and suppose $f$ is analytic on $\Delta$.
Then $f$ can be approximated on $\Delta$ by polynomials,
with exponential convergence as $n\to\infty$, where $n$ is
the degree, and if $f$ is not entire, then the asymptotically
optimal rate of convergence is $\Oe(\kern .5pt \rho^{-n})$,
where $\rho $ is the radius of the disk of analyticity of $f$
about $z=0$.\footnote{We use the symbol $\Oe$ defined as follows:
$g(n) = \Oe(\kern .5pt \rho^{-n})$ if for all $\varepsilon>0$,
$g(n) = O((\kern .5pt \rho -\varepsilon)^{-n})$ as $n\to\infty$.
This enables us to focus on exponential convergence rates without
being distracted by lower-order algebraic terms.} Truncated
Taylor series achieve this rate, as do interpolants in roots of
unity or in any other point sets uniformly distributed on the
unit circle as $n\to\infty$~\cite{gaier,walsh}.  Algorithms for
computing polynomial interpolants in roots of unity are discussed
in~\cite{akt}.  By the maximum modulus principle, any polynomial
approximation on the unit circle will have the same maximum error
over the unit disk.

\begin{figure}
\begin{center}
\vskip .2in
\includegraphics[scale=.6]{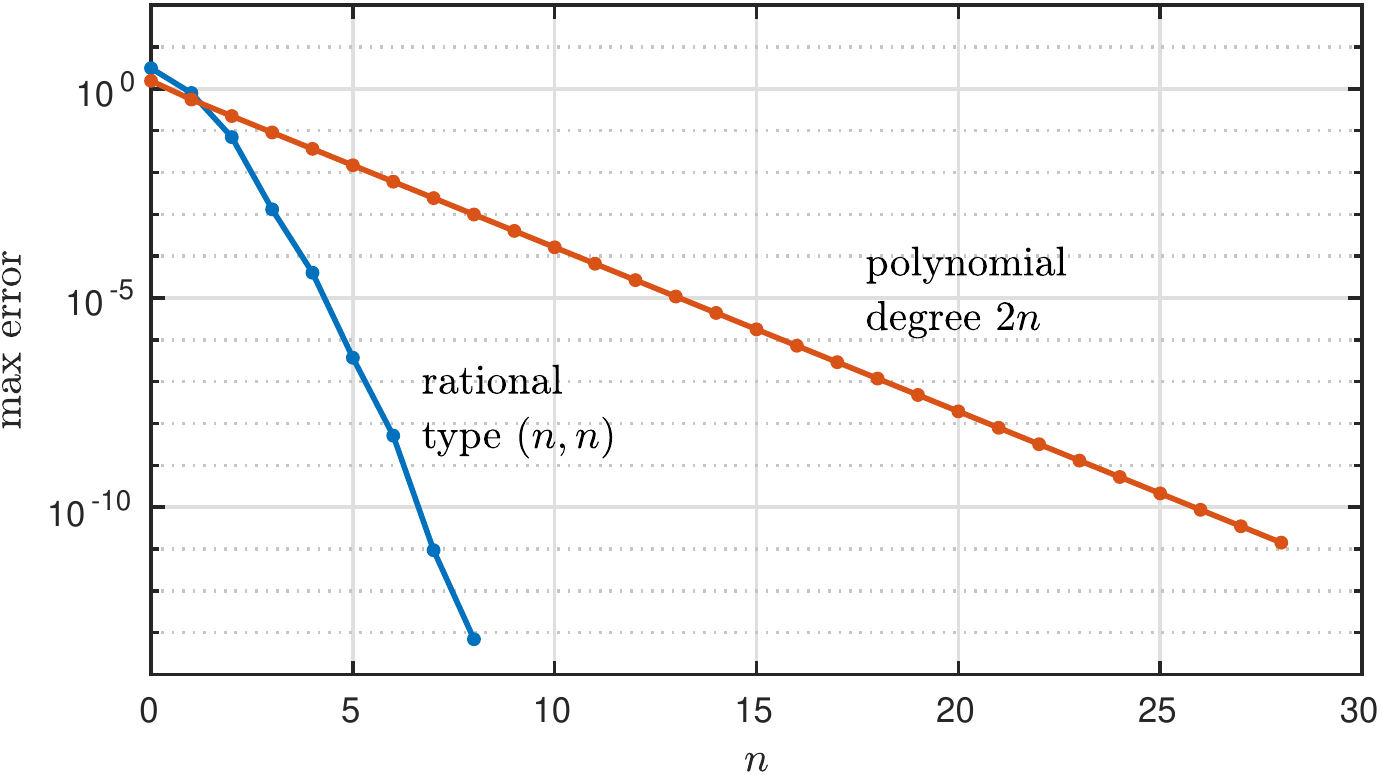}
\end{center}
\caption{\label{figapp1} Application $\ref{sec-applics}.1$.\ \ Even
for an analytic function on
the unit disk, rational approximations may be more
efficient than polynomials.  Here $f(z) = \tan(z)$ is
approximated by the AAA algorithm in $128$ points on the unit circle.
In this and subsequent figures, ``max error''
refers to the maximum error over the discrete approximation set $Z$.}
\end{figure}

Rational approximations may be far more efficient, however,
depending on the behavior of $f$ outside $\Delta$.  For example,
consider the function $f(z) = \tan(z)$.  Because of the poles
at $\pm \pi /2 \approx 1.57$, polynomial approximations can
converge no faster than $\Oe(1.57^{-n})$, so that $n\ge 52$
will be required, for example, for 10-digit accuracy.  Rational
approximations can do much better since they can capture poles
outside the disk.  Figure~\ref{figapp1} shows maximum-norm
approximation errors for polynomial and type $(\nu,\nu)$ rational
AAA approximation of this function in $M = 128$ equispaced points
on the unit circle.  (The results would be much the same with
$M=\infty$.)  We see that the rational approximation of type
$(7,7)$ is more accurate than the polynomial of degree~$52$.

In principle, there is no guarantee that AAA rational
approximations must be free of poles in the interior of $\Delta$.
In this example no such poles appear at any step of the iteration,
however, so the errors would be essentially the same if measured
over the whole disk.  Incidentally, one may note that since $f$
is an odd function with no poles in the disk, it is natural to
approximate it by rational functions of type $(\kern .5pt \mu,\nu)$
with $\nu$ even.  This is approximately what happens as the
algorithm is executed, for each approximation of type $(m-1,m-1)$
with $m-1$ odd has a pole close to $\infty$; in that sense the
rational function is nearly of type $(m-1,m-2)$.  For $m-1 = 3$
there is a pole at $5.4\times 10^{14}$, for $m-1 = 5$ there is a
pole at $8.8\times 10^{11}$, and so on.  The remaining poles are
real numbers lying in approximate plus/minus pairs about $z=0$,
with the inner pair closely approximating $\pm \pi/2$ for $m-1 =
6$ and higher.  See Section~\ref{sec-symm} for further remarks
about approximating even and odd functions.

This experiment confirms that rational functions may have
advantages over polynomials on the unit disk, even for
approximating analytic functions.  A different question is, how
does the AAA approximant compare with rational approximations
computed by other methods such as the {\tt ratdisk} algorithm
of~\cite{gpt}?  For this example, other algorithms can do about
equally well, but the next example is different.
In Section~\ref{comparisons} we make comparisons with the algorithm
known as vector fitting.

\subsection{\label{sec-app2}Analytic functions with nearby branch points;
comparison with ratdisk}
When the singularities of a function $f(z)$ outside of the disk
are just poles, a rational approximation on the unit circle can
``peel them off'' one by one, leading to supergeometric convergence
as in Figure~\ref{figapp1}.  Branch point singularities are not so
simple, but as is well known, rational approximation is effective
here too.  For example, Figure~\ref{figapp2} shows polynomial
and AAA rational approximations to $f(z) = \log(1.1-z)$ in 256
roots of unity.  The degree $2n$ polynomials converge at the rate
$\Oe(1.1^{-2n})$, while the rational approximations converge at
a rate $\Oe(\rho^{-n})$ with $\rho \approx 9.3$.
The precise constant for best rational approximations can be
determined by methods of potential theory due to Gonchar,
Parfenov, Prokhorov, and Stahl and summarized in the final paragraph
of~\cite{levinsaff}.  In fact $\rho = R^2$, where $R\approx 3.046$
is the outer radius of the annulus with inner radius~$1$ that is conformally
equivalent to the domain exterior to the unit disk and the slit $[1.1,\infty)$.
A formula for $R$ is given on p.~609 of~\cite{gk}:
\begin{equation}
R = \exp\Big({\pi K(\sqrt{1-\kappa^2}\kern 1pt)\over 4 K(\kappa)}\Big), \quad
\kappa = (c-\sqrt{c^2-1}\kern 1pt)^2, \quad c = 1.1,
\label{elliptic}
\end{equation}
where $K(\kappa)$ is the complete elliptic integral of the first kind.

\begin{figure}
\begin{center}
\vskip .2in
\includegraphics[scale=.60]{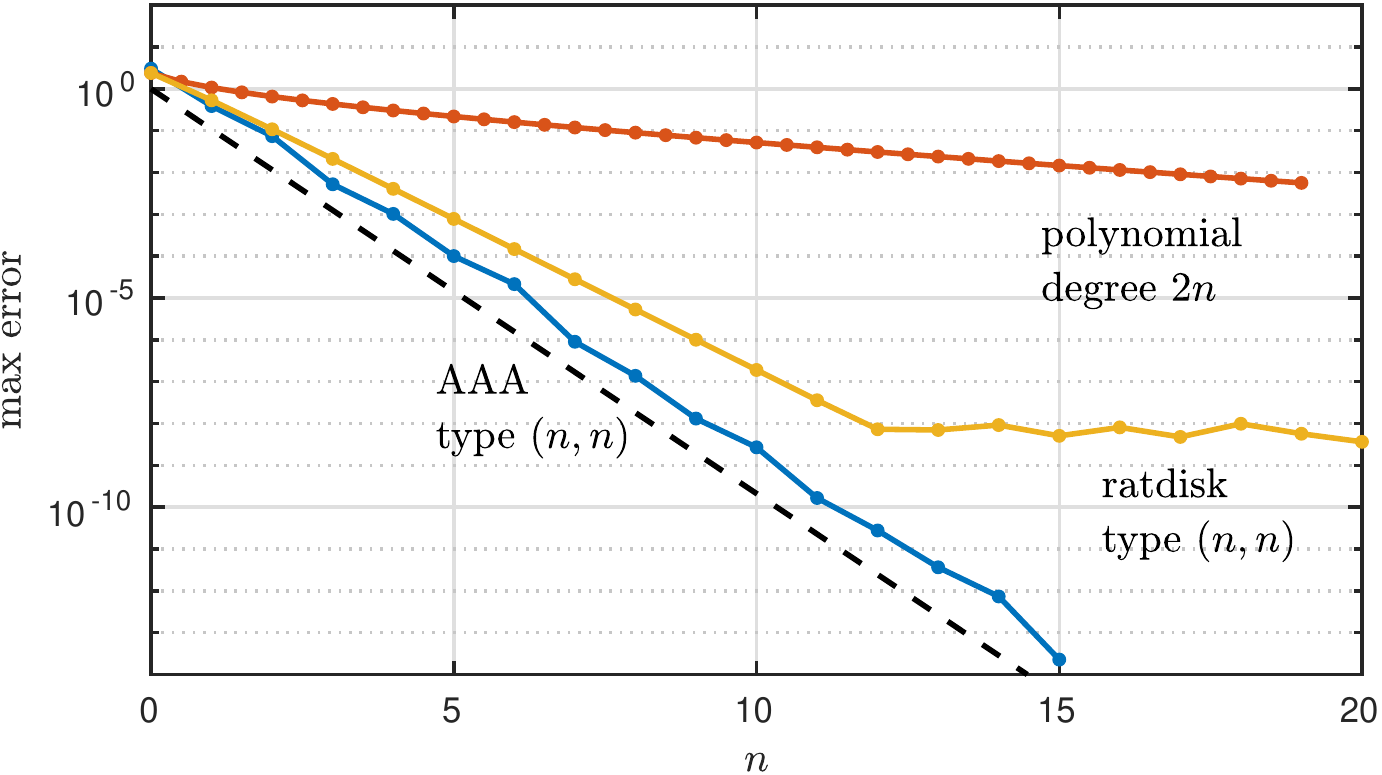}
\end{center}
\caption{\label{figapp2} Application $\ref{sec-applics}.2$.\ \ Rational functions are
also effective for approximating $f(z) = \log(1.1-z)$ in $256$ points
on the unit circle.  This function has a branch point outside
the unit disk instead of poles.  The slope of the dashed line marks the
convergence
rate for best type $(n,n)$ approximations on the unit disk.}
\end{figure}

It is well known that when rational functions approximate functions
with branch cuts, most of the poles tend to line up along branch
cuts~\cite{aptek,stahl}.  This effect is shown in Figure~\ref{phasefig}.
Apart from
the zero near $0.1$ that matches the zero there of $f$,
the zeros and poles of $r$ interlace along a curve close to
the branch cut $[1,\infty)$.  The fact that they end up
slightly below rather than above the cut
is of no significance except in illustrating that the core AAA
algorithm does not impose a real symmetry. 
As it happens, for this problem the first support point
is $z_1^{}=1$, the second is $z_2^{} = -1$, and then the real symmetry is
broken at step 3 with the selection of $z_3^{}\approx 0.87-0.49i$.
\begin{figure}
\begin{center}
\vskip .4in
\includegraphics[scale=.75]{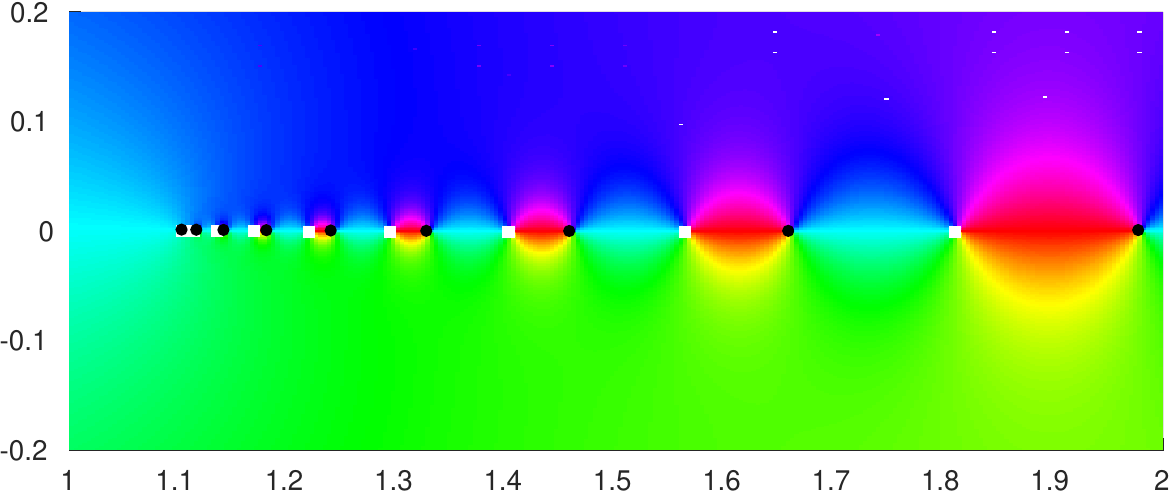}
\end{center}
\caption{\label{phasefig} Application $\ref{sec-applics}.2$, continued.
A phase portrait~{\rm \cite{wegert}} of the
rational approximant\/ $r$ shows zeros (black dots) and 
poles (white squares) interlacing along the branch cut $[1,\infty)$.}
\end{figure}

Figure~\ref{figapp2} also includes a third curve showing
convergence for a technique based on representation of a
rational function by a quotient of polynomials $p(z)/q(z)$
--- the {\tt ratdisk} approximant of~\cite{gpt}, computed
by linearized least-squares minimization of $\|f q -
p\|_Z^{}$ where $Z$ consists of the $256$ roots of unity.
(In Chebfun we calculate this with{\tt [p,q,r] = ratinterp(@(z)
log(1.1-z),n,n,256,'unitroots')}.)  This linearized problem uses
bases $1, z, z^2,\dots$ in the numerator and the denominator,
whereas AAA uses partial fraction bases, and that is why the
rational approximants need not match even for small $n$, when
rounding error is not an issue.  For larger~$n$, {\tt ratdisk}
additionally encounters numerical instability because its numerator
$p$ and denominator $q$ both become very small for values of $z$
near the branch cut, leading to cancellation error.

\subsection{\label{sec-app3}Meromorphic functions in the unit
disk from boundary values}\ Now suppose $f$ is analytic on the
unit circle and meromorphic in $\Delta$, which means, it is
analytic in $\Delta$ apart from poles, which will necessarily be
finite in number.  If $\tau$ is the number of poles counted with
multiplicity, then $f$ can be approximated on the unit circle by
rational functions of type $(\kern .5pt\mu,\nu)$ as $\mu\to\infty$
for any $\nu \ge \tau$.  For $\nu < \tau$ this is not possible in
the sense of approximation to arbitrary accuracy, but that is not
the whole story, because one may be interested in a fixed accuracy
such as 8 or 16 digits.  Suppose for example that $f(z)$ is
nonzero for $|z|=1$ with winding number $-q$ there for some $q>0$,
which implies that the number of poles in the unit disk exceeds
the number of zeros by $q$.  Then no type $(\kern .5pt\mu,\nu)$
rational function $r$ with $\nu < q$ can differ from $f$ on $|z|=1$
by less than $\min_{|z|=1}^{} |f(z)|$, because such an $r$ would
have to have winding number $-q$ too (this is Rouch\'e's theorem).
On the other hand for $\nu\ge q$, a rational approximation can
get the winding number right and good accuracy may be possible.

\begin{figure}
\begin{center}
\vskip .2in
\includegraphics[scale=.6]{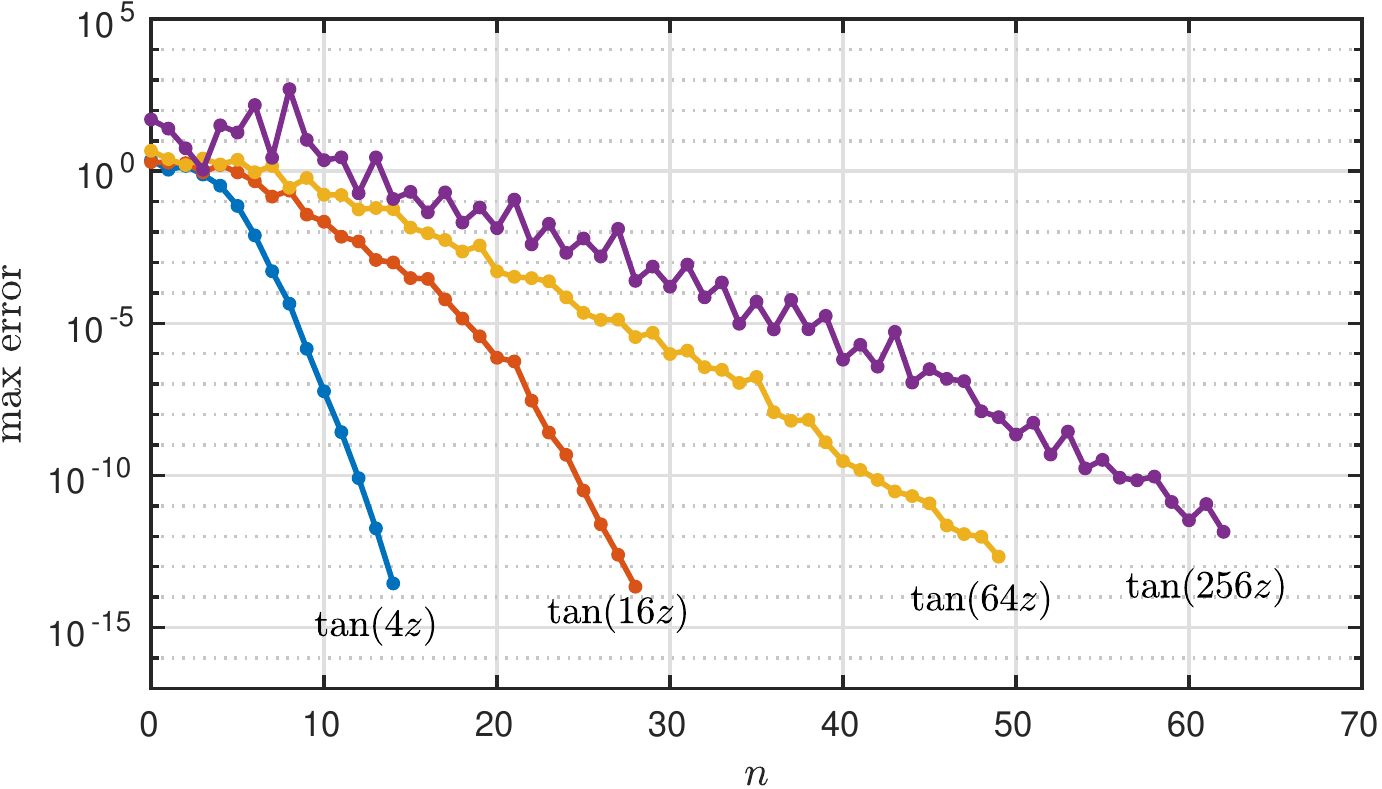}
\end{center}
\caption{\label{figapp3} Application $\ref{sec-applics}.3$.\ \ Errors in
type $(n,n)$ AAA approximation of\/ $\tan(\beta z)$ with
$\beta  = 4, 16, 64, 256$ in\/ $1000$ points on the unit circle.}
\end{figure}

For example, Figure~\ref{figapp3} shows errors in AAA approximation
of $\tan(\beta z)$ with $\beta = 4, 16, 64, 256$ in 1000 equispaced
points on the unit circle.  Robust superlinear convergence is
observed in all cases.  Now these four functions have 2, 10, 40,
and 162 poles in the unit disk, respectively, and from the figure
we see that they are approximated on the unit circle to 13-digit
accuracy by rational functions of types $(14,14)$, $(28,28)$,
$(49,49)$ and $(62,62)$.  Clearly the fourth approximation, at
least, cannot be accurate throughout the unit disk, since it does
not have enough poles.  The right image of Figure~\ref{figapp3b}
confirms that it is accurate near the unit circle, but not
throughout the disk.  By contrast the left image of the figure,
for the simpler function~$f$ with $\beta=16$, shows good accuracy
throughout the disk.  This example highlights fundamental facts
of rational approximation: one cannot assume without careful
reasoning that an approximation accurate on one set must be
accurate elsewhere, even in a region enclosed by a curve,
nor that it provides information about zeros or poles of the
function being approximated away from the set of approximation.
An approximation on a closed curve does however tell you the
winding number, hence the {\em difference} between numbers of
poles and zeros of the underlying function in the region enclosed.

\begin{figure}
\vskip .2in
\begin{center}
\includegraphics[scale=.8]{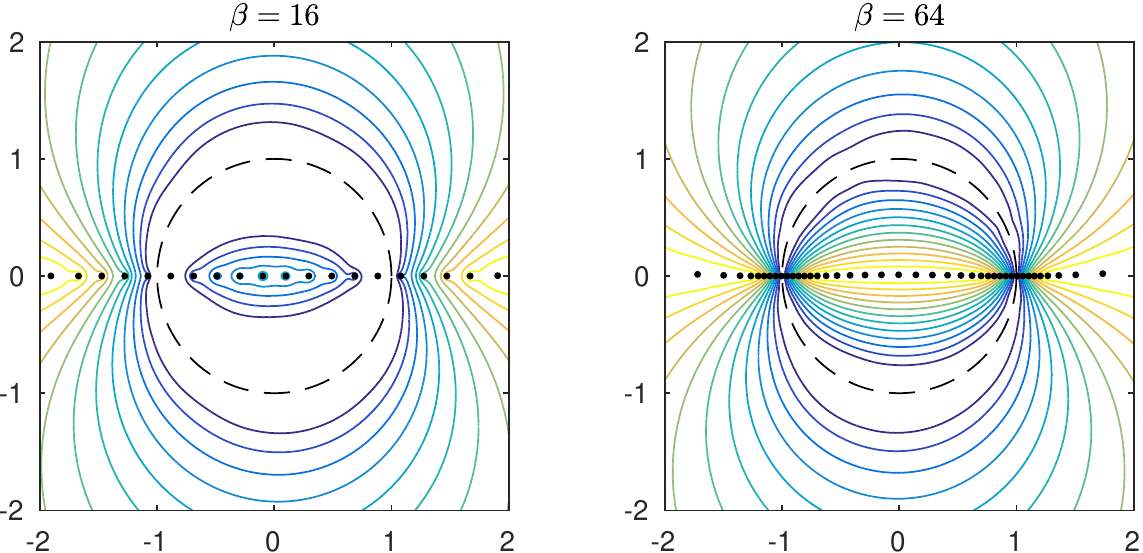}
\end{center}
\caption{\label{figapp3b} Application $\ref{sec-applics}.3$ continued, showing
more information about
the final approximations\/ $r$ of Figure~$\ref{figapp3}$ in
the cases $\beta =16$ and\/ $64$.  Contours of the error 
$|f(z) - r(z)|$ mark powers of $10$
from yellow (\/$10^{-1}$) to dark blue (\/$10^{-12}$).
With\/ $\beta =16$, all the poles of\/ $f$ are approximated by poles
of\/ $r$ and $r\approx f$ throughout $\Delta$.  
With\/ $\beta =64$, $f$ has more poles than $r$, and\/ $r$
approximates $f$ on the circle but not throughout $\Delta$.}
\end{figure}

We have spoken of the possibility of accuracy ``throughout
$\Delta$.''  There is a problem with such an expression, for even
if the poles of $\kern .5pt r$ in the unit disk match those of
$f$ closely, they will not do so exactly.  Thus, mathematically
speaking, the maximal error in the disk will almost certainly
be $\infty$ for all rational approximations of functions that
have poles.  Nevertheless one feels that it ought be be possible
to say, for example, that $1/(z-10^{-15})$ is a good approximation
to $1/z$.  One way to do this is to measure the difference between
complex numbers $w_1^{}$ and $w_2^{}$ not by $|w_2^{}-w_2^{}|$
but by the chordal metric on the Riemann sphere.  This and other
possibilities are discussed in~\cite{blm}.

\subsection{\label{sec-app4}Meromorphic functions in the unit disk from
boundary and interior values}\ If the aim is approximation
throughout the unit disk, the simplest approach is to distribute
sample points inside the disk as well as on the circle,
assuming function values are known there.  Figures~\ref{figapp4}
and~\ref{figapp4b} repeat Figures~\ref{figapp3} and~\ref{figapp3b},
but now with $3000$ points randomly distributed in the unit disk
in addition to $1000$ points on the unit circle.  Note that
$n$ must now be larger for convergence with $\beta =64$ and
especially $256$.  (This last case leads to large matrices and
the only slow computation of this article, about 30 seconds
on our laptops.)  We mentioned above that $f$ has 2, 10, 40,
and 162 poles in $\Delta$ for the four values of $\beta $.
The AAA approximants match the first three of these numbers, but
the fourth approximation comes out with 164 poles in $\Delta$,
the two extra ones being complex numbers near the real axis,
Froissart doublets of the non-numerical variety, with residues on
the order of $10^{-7}$ and $10^{-11}$.  If the number of points
in the disk is doubled to 6000, they go away.

\begin{figure}
\begin{center}
\vskip .2in
\includegraphics[scale=.6]{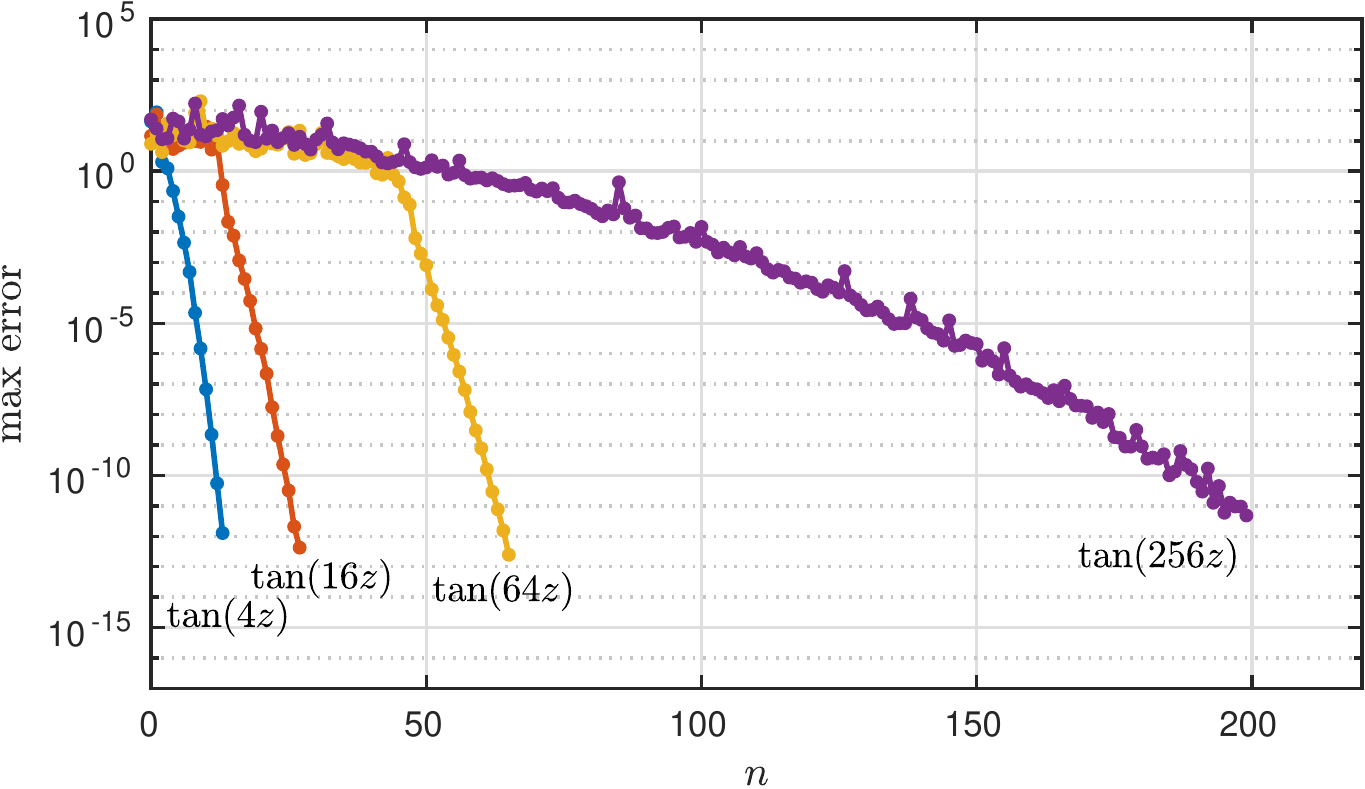}
\end{center}
\caption{\label{figapp4} Application $\ref{sec-applics}.4$.
Repetition of Figure~$\ref{figapp3}$
for a computation with $3000$ points in the interior of the unit disk
as well as $1000$ points on the boundary.  Higher values
of\/ $n$ are now needed to achieve $13$-digit accuracy, but all
the poles are captured.
If many more than $3000$ points were used,
the final curve would not show any convergence until $n\ge 162$.}
\end{figure}

\begin{figure}
\begin{center}
\vskip .2in
\includegraphics[scale=.8]{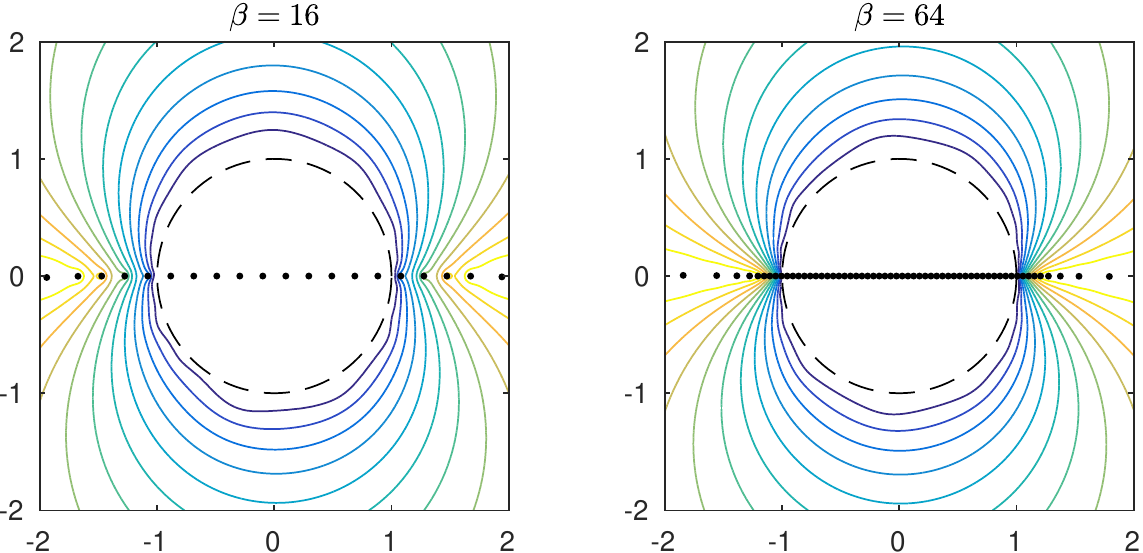}
\end{center}
\caption{\label{figapp4b} Application $\ref{sec-applics}.4$, continued,
a repetition of Figure~$\ref{figapp3b}$
for the computation with points in the interior of the disk.
Now the rational approximants are accurate throughout $\Delta$.}
\end{figure}

\subsection{\label{sec-app5}Approximation in other connected
domains} The flexibility of AAA approximation becomes most
apparent for approximation on a non-circular domain~$Z$.  Here
the observations of the last few pages about the advantages of
barycentric representations compared with quotients of polynomials
$p(z)/q(z)$ still apply.  In addition, the latter representations
have the further difficulty on a non-circular domain $Z$ that
it is necessary to work with a good basis for polynomials on $Z$
to avoid exponential ill-conditioning.  If $E$ is connected, the
well-known choice is the basis of Faber polynomials, which might
be calculated a priori or approximated on the fly in various
ways~\cite{gaier}.  The AAA algorithm circumvents this issue
by using a partial fraction basis constructed as the iteration
proceeds.  For example, Figure~\ref{figapp5} shows error contours of the
AAA approximant of the reciprocal Bessel function $1/J_0^{}(z)$
in 2000 uniformly distributed random points in the rectangle
defined by the corners $\pm i$ and $10 \pm i$.  The algorithm
exits at step $m=13$ with an approximation of type $(12,12)$.
The computed poles in the domain rectangle are almost exactly
real, even though no real symmetry has been imposed, and almost
exactly equal to the poles of $f$ in this region.

{\small
\begin{verbatim}

Poles of f in rectangle     Poles of r in rectangle
2.404825557695780           2.404825557695776 - 0.000000000000001i
5.520078110286327           5.520078110286310 - 0.000000000000000i
8.653727912911013           8.653727912911007 + 0.000000000000002i

\end{verbatim}
\par}

\begin{figure}
\begin{center}
\vskip .2in
\includegraphics[scale=.75]{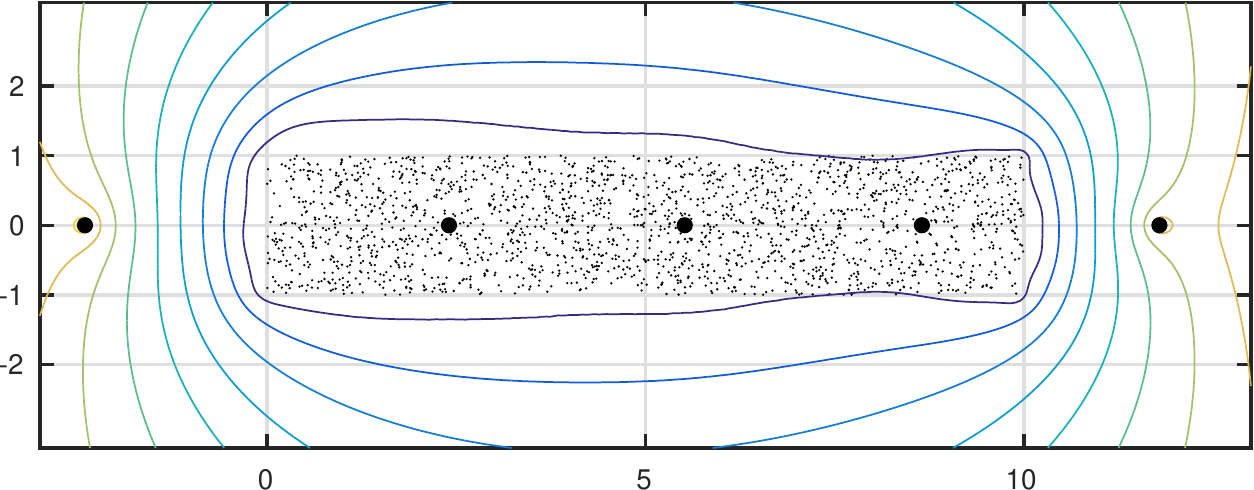}
\end{center}
\caption{\label{figapp5} Application $\ref{sec-applics}.5$.
Approximation of\/ $1/J_0^{}(z)$ in\/ $2000$
random points in a rectangle in the complex plane.  Contours
mark levels $10^{-12},\dots, 10^{-1}$ of the error
$|f(z)-r(z)|$, as in Figures~$\ref{figapp3b}$ and\/~$\ref{figapp4b}$, and the
dots are the poles of the rational approximant.}
\end{figure}

Figure~\ref{sec11fig} in Section~\ref{comparisons} compares the condition numbers
of the basis in this example against what would be found for
other methods.

\subsection{\label{sec-app6}Approximation in disconnected domains}
One of the striking features of rational functions is their
ability to approximate efficiently on disconnected domains and
discretizations thereof.  (This is related to the power of IIR as
opposed to FIR filters in digital signal processing~\cite{dsp}.) As
an example, Figure~\ref{figapp6} shows error contours for
the approximation of $f(z) = \hbox{sign}(\hbox{Re}(z))$ on a set $Z$
consisting of 1000 points equally spaced around a square in the
left half-plane (center $-1.5$, side length $2$) and 1000 points
equally spaced on a circle in the right half-plane (center $1.5$,
diameter~$2$).  This is the only one of our applications in which
numerical Froissart doublets appear.  Convergence is achieved
at $m = 51$, and then six doublets are removed by {\tt cleanup}.
One would expect the resulting approximant to be of type $(44,44)$,
but in fact the type is $(43,43)$ because one weight $w_j^{}$
is identically zero as a consequence of the Loewner matrix having
half its entries exactly zero.

\begin{figure}
\vskip .3in
\begin{center}
\includegraphics[scale=.6]{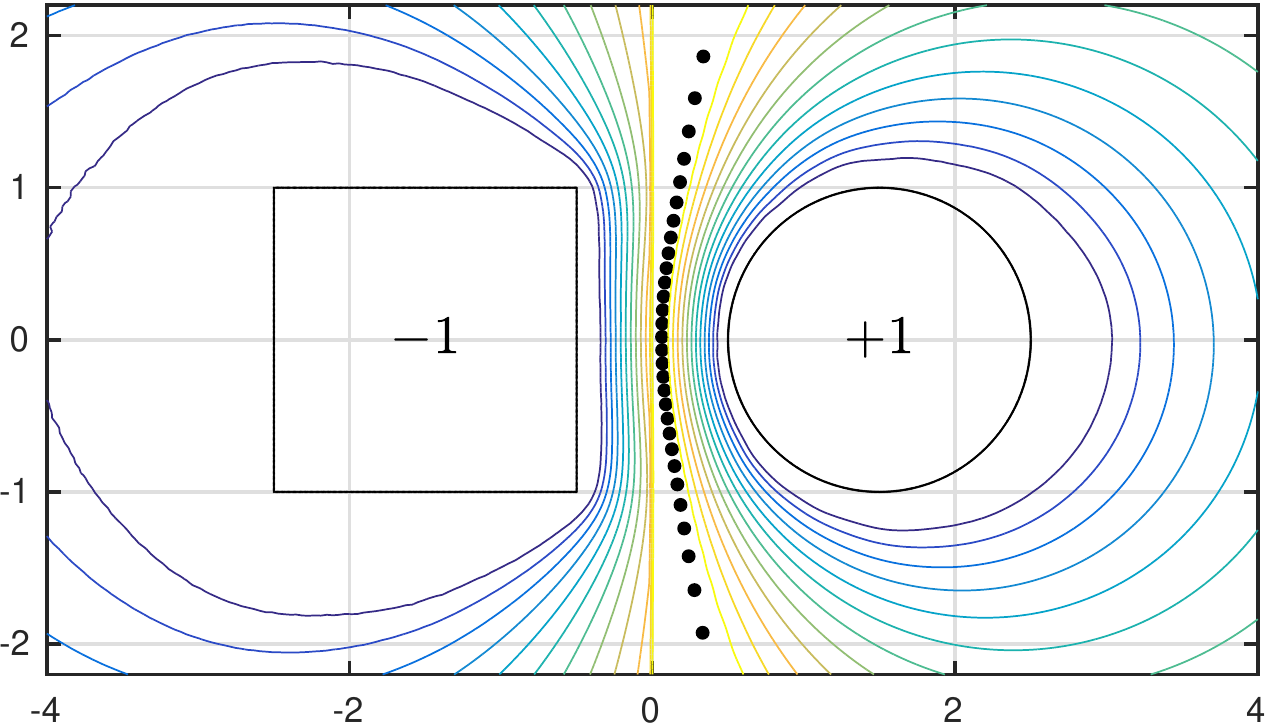}
\end{center}
\caption{\label{figapp6} Application $\ref{sec-applics}.6$.
Approximation of\/ $\hbox{\rm sign}(\hbox{\rm Re}
(z))$ in\/ $2000$
points on a disconnected set consisting
of a square and a circle in the complex plane, with contour
levels as before.  Six numerical Froissart doublets have been automatically
removed.}
\end{figure}

There is not much literature on methods for practical
rational approximation on disconnected domains.  If one uses
a representation of $r$ as a quotient of polynomials $r(z) =
p(z)/q(z)$, the problem of finding well-conditioned bases for $p$
and $q$ arises again.  The generalizations of Faber polynomials for
disconnected domains are known as {\em Faber--Walsh polynomials,}
but they seem to have been little used except in~\cite{sl}.

\subsection{\label{sec-app7}\boldmath Approximation of $|x|$ on $[-1,1]$}

Our next application concerns the first of the ``two famous
problems'' discussed in Chapter 25 of~\cite{atap}: the rational
approximation of $f(x) = |x|$ on $[-1,1]$.  In the 1910s it was
shown by Bernstein and others that degree $n$ polynomial approximants of this
function achieve at best $O(n^{-1})$ accuracy, whereas a celebrated
paper of Donald Newman in 1964 proved that rational approximations
can have root-exponential accuracy~\cite{newman}.  The precise
convergence rate for minimax (best $L^\infty$)
approximations was eventually shown to be $E_{nn}^{}(|x|)
\sim 8\exp(-\pi\sqrt n\kern .7pt )$~\cite{stahlabs}.

Computing good rational approximations of $|x|$ is notoriously
difficult.  The trouble is that good approximations need
poles and zeros exponentially clustered in conjugate pairs
along the imaginary axis near $x=0$, with the extrema of the
equioscillating error curve similarly clustered on the real axis.
Varga, Ruttan, and Carpenter needed 200-digit extended precision to
compute minimax approximations up to type $(80,80)$ using $p/q$
representations~\cite{vrc}.  In standard 16-digit arithmetic,
Chebfun's Remez algorithm code used to fail at degree about
$(10,10)$, like other available Remez codes.  However,
since the present paper was first submitted for publication,
the algorithm has been redesigned based on barycentric
representations with adaptive support points, and the picture has
changed completely.  The new {\tt minimax} gets successfully
up to type $(80,80)$ in ordinary floating-point arithmetic~\cite{fntb}.
See also Chapter~3 of~\cite{ionita}.

\begin{figure}
\begin{center}
\vskip .2in
\includegraphics[scale=.60]{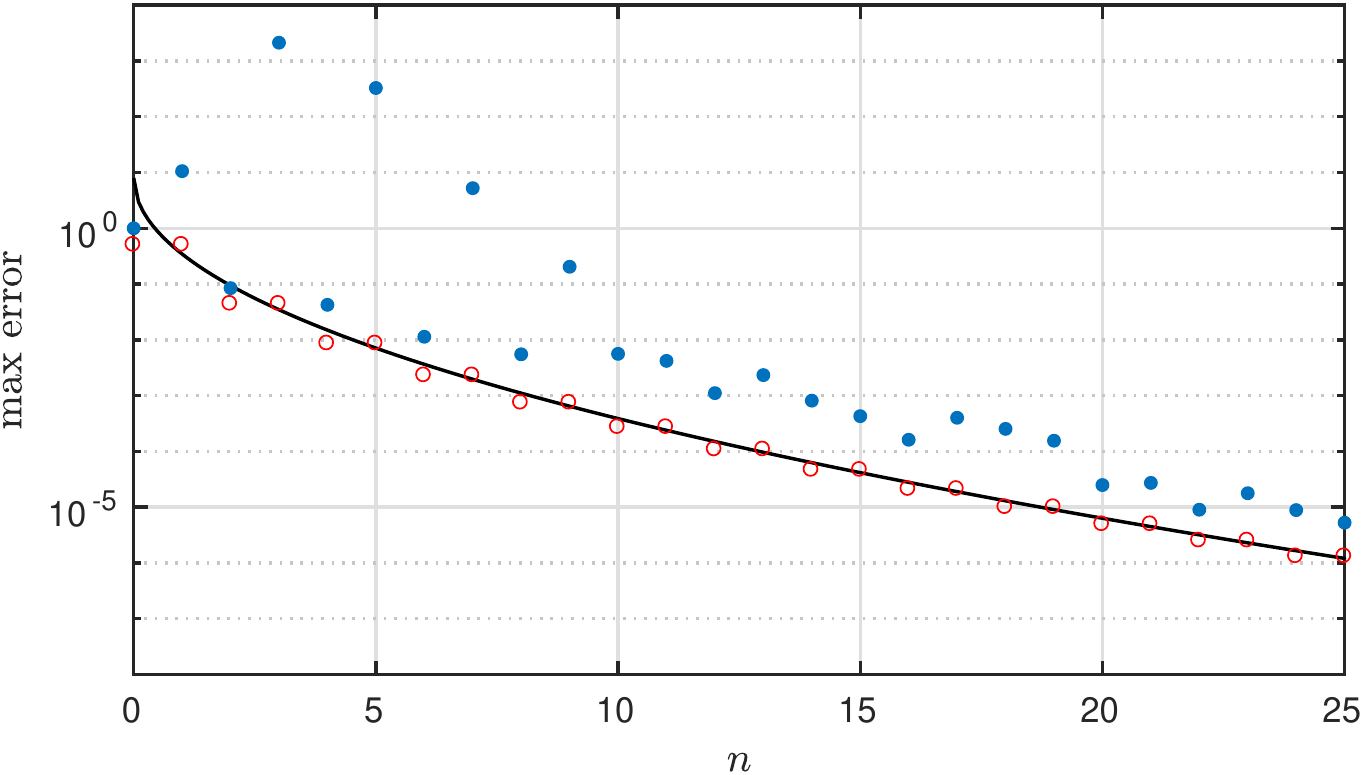}
\end{center}
\caption{\label{figapp7}Application $\ref{sec-applics}.7$.
Upper dots: errors in AAA approximation
of\/ $|x|$ in\ {\rm 200,000} points in $[-1,1]$.  Odd values of $n\ge 3$
give approximations
with poles in $[-1,1]$, while the even-$n$ approximations are
pole-free for $n\le 6$.
Because of the large value of $M$,
this computation took about\/ $8$ seconds in Matlab on a desktop machine.
Lower dots: best approximation errors for the same problem, superimposed
on the asymptotic result of\/~{\rm \cite{stahlabs}}.}
\end{figure}

Here, we show what can be achieved with AAA itself.  Suppose we
apply the AAA algorithm with a set $Z$ of 200,000 equispaced
points in $[-1,1]$.  The results in Figure~\ref{figapp7} show
interesting behavior.  The approximations of type $(n,n)$ with
$n = 0,1,2$ are as usual, with the error at $m=2$ not far from
optimal.  The approximations with $n = 3,5,7,\dots,$ however,
have large errors on the sample set $Z$ and in fact infinite
errors on $[-1,1]$.  In these cases AAA must use an odd number
of poles in approximating an even function, so at least one of
them must be real, and it falls in $[-1,1]$.  These Froissart
doublets are not numerical: they would appear even in exact
arithmetic.  Eventually the dots cease to fall on two clear curves
distinguishing even and odd $n$, and in fact, the approximations
from $n=8$ on all have at least one pole in $[-1,1]$, though this
is not always apparent in the data of the figure.

To approximate $|x|$ on $[-1,1]$, one can do better by
exploiting symmetry and approximating $\sqrt x$ on $[\kern
.5pt 0,1]$, an equivalent problem.  (See p.~213 of~\cite{atap}.
This was the formulation
used by Varga, Ruttan, and Carpenter~\cite{vrc}.)
This transformation enables successful AAA approximations up to $n=80$,
better than 10 digits of accuracy, with no poles in the interval
of approximation.  This is highly satisfactory as regards the $|x|$
problem, but it is not a strategy one could apply in all cases.
The fact is that the core AAA algorithm risks introducing unwanted
poles when applied to problems involving real functions on real
intervals, and the reason is not one of even or odd symmetry.

\subsection{\label{sec-app8}\boldmath
Approximation of $\exp(x)$ on $(-\infty,0\kern .5pt]$}
We now look at the second of
the ``two famous problems'' of~\cite{atap}: the rational approximation of
$f(x) = \exp(x)$ on $(-\infty,0\kern .5pt]$.
As told in~\cite{atap},
many researchers have contributed to this problem over the
years, including Cody, Meinardus and Varga, Trefethen and Gutknecht,
Magnus, Gonchar and Rakhmanov, and Aptekarev.  The sharp result is
that minimax approximation errors for type $(\nu,\nu)$
approximation decrease geometrically
as $\nu\to\infty$ at a rate $\sim 2H^{n+1/2}$, where
$H$, known as {\em Halphen's constant,} is approximately
$1/9.28903$~\cite{aptek1}.  The simplest effective way
to compute these approximants accurately is to transplant
$(-\infty,0\kern .5pt ]$ to $[-1,1]$ by a
M\"obius transformation and then apply CF approximation~\cite{tws}.
Figure~\ref{figapp8} shows that one can come within a
factor of~10 of optimality by applying
the AAA algorithm with $Z$ as the set of 4000 points logarithmically
spaced from $-10^4$ to $-10^{-3}$.  (We loosened the tolerance
from $10^{-13}$ to $10^{-12}$.)  Note that in this case the
AAA algorithm is working with little trouble on
a set of points ranging in amplitude by seven orders of magnitude.
\begin{figure}
\begin{center}
\vskip .2in
\includegraphics[scale=.6]{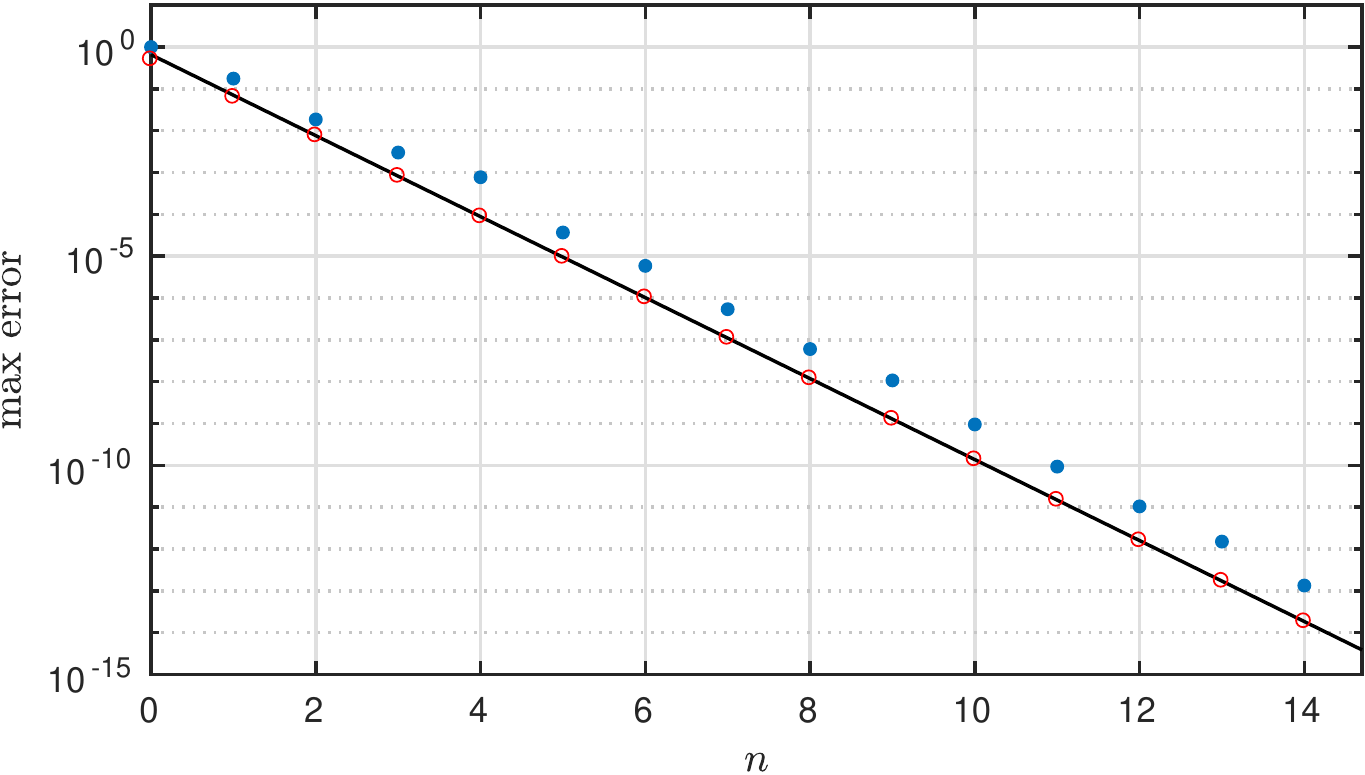}
\end{center}
\caption{\label{figapp8}Application $\ref{sec-applics}.8$.
Upper dots: errors in AAA approximation of
$\exp(x)$ in\/ $2000$ points in $(-\infty,0\kern .5pt]$ logarithmically
spaced from $-10^4$ to $-10^{-3}$.
This computation took about\/ $1$ second in Matlab on a desktop machine.
Lower dots: best approximation errors for the same problem, superimposed
on the asymptotic result of~{\rm \cite{aptek1}}.}
\end{figure}

\begin{figure}
\begin{center}
\vskip .2in
\includegraphics[scale=.65]{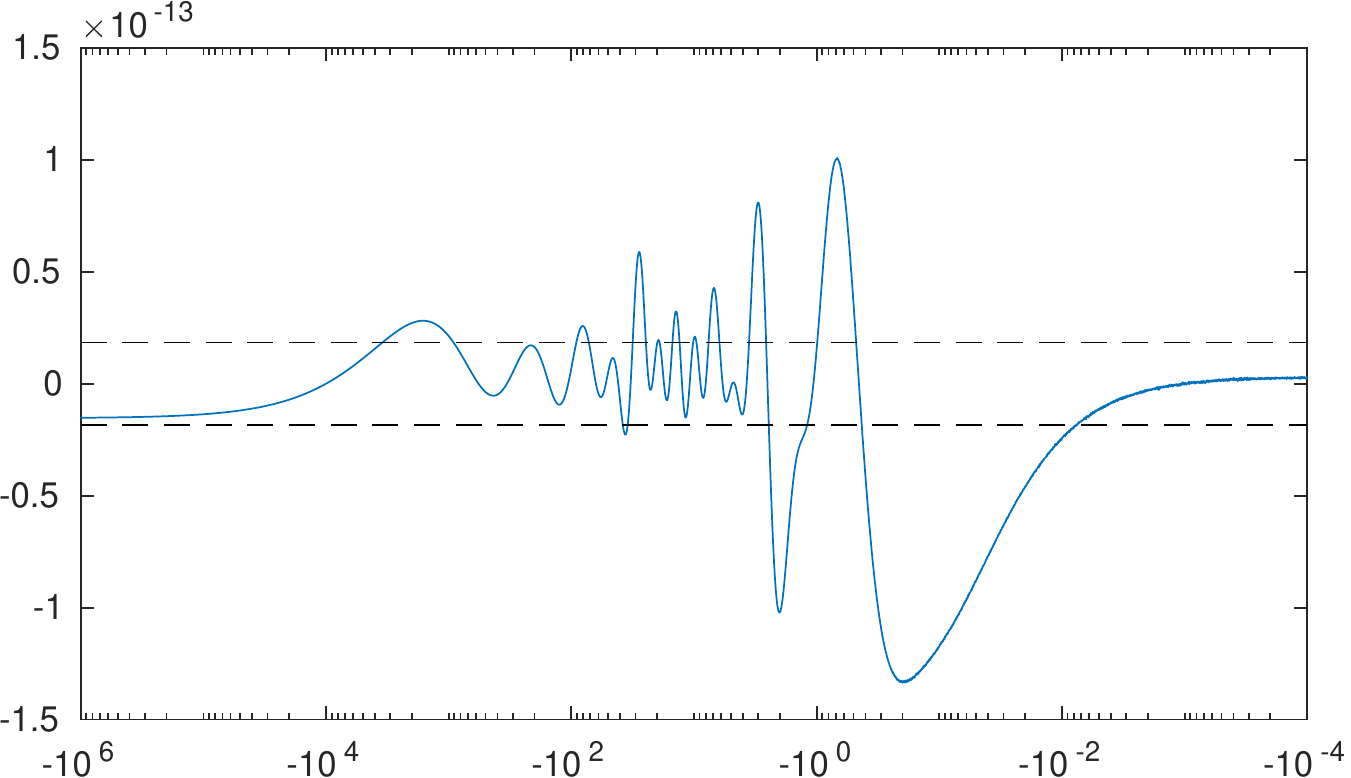}
\end{center}
\caption{\label{figapp8b}Application $\ref{sec-applics}.8$, continued.  Error curve
$f(x)-r(x)$, $x\in (-\infty, 0\kern .5pt ]$,
for the final approximant, of type $(14,14)$, which
not-quite-equioscillates
between $26$ local extrema of alternating signs.  This approximation is
of type $(14,14)$, so the error curve for the
minimax rational function would have $30$ equioscillatory extreme
points, with the amplitude shown by the dashed lines.  See~{\rm \cite{fntb}}
for next steps to bring such an error curve to truly equioscillatory form.}
\end{figure}

\subsection{\label{sec-app9}Clamped beam model from
Chahlaoui and Van Dooren}
Our final example is the clamped beam model
from the NICONET collection of examples for model order reduction by
Chahlaoui and Van Dooren~\cite{chdo}.  Here $f$
is a rational function of type $(348,348)$ of the
form $f(z) = c^T(zI-A)^{-1} b$, where $c$ and $b$ are
given column vectors of dimension $348$ and $A$ is a 
given $348\times 348$ matrix.  As usual in model order reduction, the
aim is to approximate $f$ on the imaginary axis.  

As AAA iterates for this problem, it captures one pole pair
after another successfully.  We sampled $f$ at $500$ logarithmically
spaced points from $10^{-2}\kern .5pt i$ to $10^2\kern .5pt i$ on the imaginary axis,
as well as their complex conjugates, and 
Figure~\ref{figapp9} shows the positive halves of the resulting
rational functions after steps $m = 3$, $7$, and $13$ --- that
is, of types $(2,2)$, $(6,6)$, and $(12,12)$.  Clearly the function
is being captured successfully, one pole pair after another.
Figure~\ref{figapp9b} shows the maximum error as a function of
$n$ in a run with a specified target relative tolerance of 
$10^{-5}$.   All the poles of all the approximants lie in
the left half-plane.
\begin{figure}
\begin{center}
\vskip .2in
\includegraphics[scale=.7]{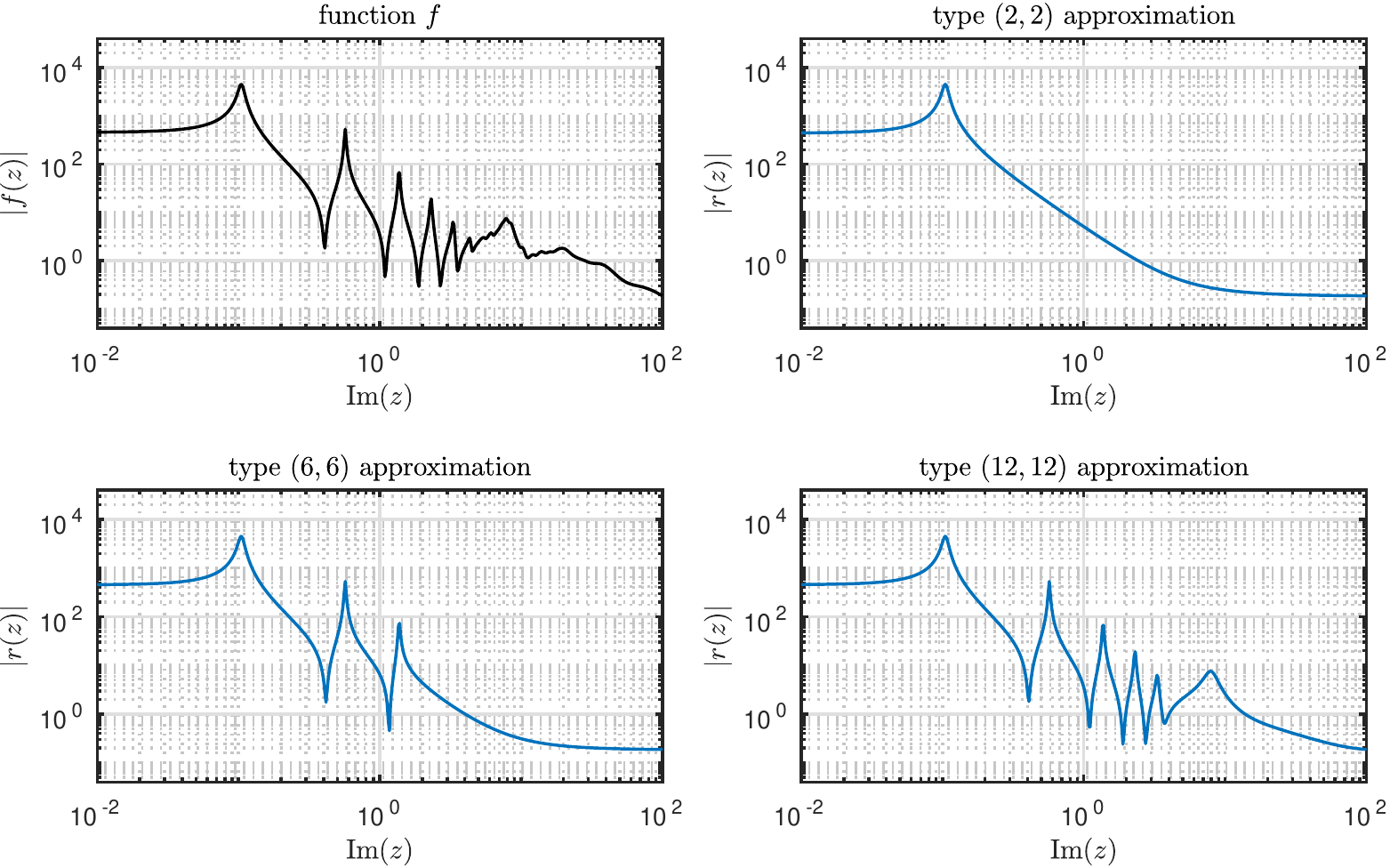}
\end{center}
\caption{\label{figapp9}Application $\ref{sec-applics}.9$.
For the clamped beam example from the collection of
Chahlaoui and Van Dooren~{\rm\cite{chdo}}, AAA captures one
pole pair after another near the imaginary axis.}
\end{figure}

\begin{figure}
\begin{center}
\vskip .2in
\includegraphics[scale=.6]{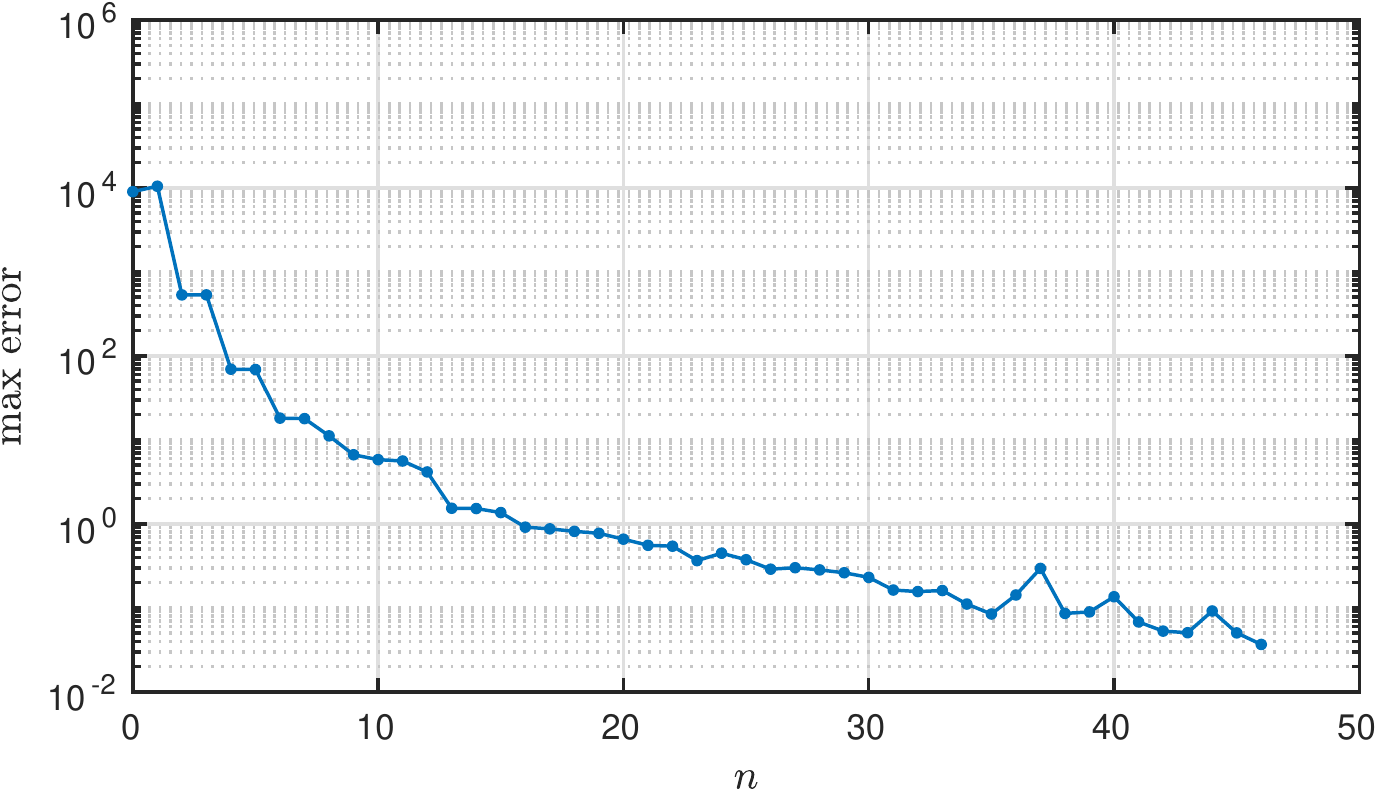}
\end{center}
\caption{\label{figapp9b}Application $\ref{sec-applics}.9$, continued.
Accuracy of the type $(n,n)$ AAA approximant 
as a function of $n$.}
\end{figure}

\section{\label{sec-symmetrized}\boldmath Modified algorithm
to treat large and small data symmetrically} On the Riemann
sphere, the point $z=\infty$ has no special status, and a function
meromorphic on the whole sphere must be rational.  From this point
of view every rational function has equal numbers of poles and
zeros: if there are $\mu$ finite zeros and $\nu$ finite poles,
then $z=\infty$ will be a pole of order $\mu-\nu$ if $\mu>\nu$
and a zero of order $\nu-\mu$ if $\nu>\mu$.  Moreover, if $r$
is a rational function with $\nu$ poles and $\nu$ zeros on the
sphere, then the same is true of $1/r$.

One may ask whether the AAA algorithm respects such symmetries.
If $r$ is a AAA approximant of $f$ of type $(m-1,m-1)$ on a
given set $Z$, will $1/r$ be the corresponding approximant of
$1/f$?  Will $r(1/z)$ be the approximant of $f(1/z)$ on $1/Z$?
The answer is yes to the second question but no to the first for
the AAA algorithm as we have formulated it, but a small change
in the algorithm will ensure this property.  In the linear
least-squares problem (\ref{l2prob}) solved at each step, the
quantity associated with the point $z$ is $|f(z) d(z) - n(z)|$.
To achieve symmetry it is enough to change this to $|d(z) -
n(z)/f(z)|$ at each point $z$ where $|f(z)|>1$.  For points
with $f(z)\ne \infty$ this amounts to dividing the row of the
matrix $A$ of (\ref{AfromC}) corresponding to $z$ by $f(z)$.
(The same scaling was introduced in~\cite{in} for reasons of
numerical stability.)  On the Riemann sphere, the northern and
southern hemispheres are now equal.  To complete the symmetry
one also replaces $f(z) - n(z)/d(z)$ by $1/f(z) - d(z)/n(z)$
at such points for the greedy choice of the next support point.

If we make the changes just described, different symmetries
are broken, the conditions of translation-invariance and
scale-invariance in $f$ of Proposition~\ref{prop}.  The source of
the problem is the constant~1 in the expressions $1/r$, $1/f$,
and $|f(z)|>1$.  (On the Riemann sphere, we have privileged the
equator over other latitudes.)  To restore scale-invariance in
$f$ one could replace the condition $|f(z)|>1$ by $|f(z)|>C$,
where $C$ is a constant that scales with $f$ such as the median
of the values $|f(Z)|$.

Besides invariance properties and elegance, another reason
for imposing a large-small symmetry is to make it possible to
treat problems where $f$ takes the value $\infty$ at some points
of $Z$.  For example, this would be necessary in the ``ezplot''
example of Figure~\ref{ezfig} if the interval $[-1.5,1.5]$ were
changed to $[-1,1]$, making $-1$ into one of the sample points.
In the symmetric setting, such a function value is no different
from any other.

\section{\label{sec-symm}\boldmath Modified algorithms to impose
even, odd, or real symmetry} Many applications involve
approximation with a symmetry: $f$ may be an even or an odd
function on a set with $Z=-Z$, or $f$ may be hermitian in
the sense of satisfying $f(\overline{z}) = \overline{f(z)}$ on a
subset of the real line or some other set with $Z = \overline{Z}$.
The core AAA algorithm will usually break such symmetries, with
consequences that may be unfortunate in practice and disturbing
cosmetically.  One may accordingly turn to variants of the
algorithm to preserve the symmetries by choosing new support
points usually in pairs, two at a time (perhaps with the point
$z=0$ treated specially).

We saw an example of even symmetry in \S\ref{sec-app7} in the
approximation of $|x|$ on $[-1,1]$, where the core AAA algorithm
produced unwanted poles.  So far as we are aware, such cases
of even or odd functions can usually be treated effectively by
changing variables from $z$ to $z^{1/2}$ rather than modifying
the AAA algorithm.

The case of hermitian symmetry is different in that one
can sometimes achieve a superficial symmetry while still not
coping with a deeper problem.  If $|x|$ is approximated on a set
$Z\subseteq [-1,2]$, for example, then all support points will of
course be real and no complex numbers will appear.  Nevertheless,
there is the problem that the AAA algorithm increases the rational
type one at a time, so that odd steps of the iteration, at least,
will necessarily have at least one real pole, which will often
appear near the singular point $x=0$.  We do not have a solution
to offer to this problem.

\section{\label{sec-munu}\boldmath Modified algorithm for
approximations of type $(\kern .5pt \mu,\nu)$} In many
applications, as well as theoretical investigations, it is
important to consider rational approximations of type $(\kern .5pt
\mu,\nu)$ with $\mu \ne \nu$.  The barycentric representation
(\ref{baryagain}) is still applicable here, with a twist.
It is now necessary to restrict the weight vector $\{w_j^{}\}$
to lie in an appropriate subspace in which the numerator
degree is constrained to be less than the denominator degree,
or vice versa.  This concept is familiar in the extreme case of
polynomial approximation, $\nu = 0$, where it is well known that
the barycentric weights must take a special form~\cite{bt}.
Barycentric representations of type $(\kern .5pt \mu,\nu)$
rational functions are a natural generalization of this
case~\cite[Sec.~3]{bbm}, and in the follow-up paper~\cite{fntb}
we give details.

The use of approximations with $\mu\ne \nu$ raises many issues.
The barycentric representations must be generalized as just
mentioned, and this adds new linear algebra features to
the algorithm.  There is also the important practical matter
of determining representations with minimal values of $\nu$
to achieve various goals.  For example, based on boundary
values of a meromorphic function on the unit circle, how best
can one determine all its poles in the disk?  The discussion
of \S\ref{sec-app3} indicated some of the challenges here.
Algorithms for minimizing the degree of the denominator have been
presented for Pad\'e approximation~\cite{ggt} and interpolation
in roots of unity~\cite{gpt}, and it would be very useful to
develop analogous extensions of AAA approximation.

\section{\label{sec-variants}\boldmath Other variants}
We have referred to the ``core'' AAA algorithm to emphasize that
in certain contexts, the use of variants would be appropriate.
Three of these have been mentioned in the last three sections.
We briefly comment now on six further variants.

{\em Weighted norms} can be introduced in the least-squares
problem at each step by scaling the rows of the matrix $\Am$ of
(\ref{nla}).  (One may also apply a nonuniform weight in the
greedy nonlinear step of the algorithm.)  This may be useful in
contexts where some regions of the complex plane need to be sampled
more finely than others, yet should not have greater weight.
For example, we used 200,000 points to get a good approximation of
$[-1,1]$ in the approximation of $|x|$ in Section~\ref{sec-app7},
but one can get away with fewer if the points are exponentially
graded but with an exponentially graded weight function to
compensate.

{\em Continuous sample sets} $Z$ such as $[-1,1]$ or the unit disk
are commonly of interest in applications, and while $Z$ is
usually discrete in practice, it is very often an approximation
to something continuous.  But it is also possible to formulate
the AAA algorithm in a continuous setting from the start, with
$M$ taking the value $\infty$.  The barycentric representation
(\ref{baryagain}) remains discrete, along with its set of support
points, but the selection of the next support point at step $m$
now involves a continuous rather than discrete optimization
problem, and the least-squares problem (\ref{l2prob}) is now a
continuous one.  (Chebfunners say that the Loewner and Cauchy
matrices (\ref{nla}) and (\ref{Cmatrix}) become $\infty\times
m$ {\em quasimatrices,} continuous in the vertical direction~\cite{chebfun}.)
Questions of weighted norms will necessarily arise whenever
$Z$ mixes discrete points with continuous components; here the
least-squares problem could be defined by a Stieltjes measure.

>From an algorithmic point of view, problems with continuous sample
sets suggest variants of AAA based on a discrete sample set $Z$
that is enlarged as the computation proceeds.  Such an approach is
used in~\cite{in} and is also an option in the Chebfun {\tt aaa}
code, but we make no claims that success is guaranteed.

{\em Confluent sample points} is the phrase we use to refer to a
situation in which one wishes to match certain derivative data as
well as function values; one could also speak of ``Hermite-AAA''
as opposed to ``Lagrange-AAA'' approximation.  Antoulas and
Anderson consider such problems in~\cite{aa}, treating them
by adding columns such as $(Z_i^{(m)}-z_m^{})^{-k}$ with $k>1$
to the Cauchy matrix (\ref{Cmatrix}) as well as new rows to the
Loewner matrix $\Am$ of (\ref{nla}) to impose new conditions.
The setting in~\cite{aa} is interpolatory, so weighting of rows
is not an issue except perhaps for numerical stability, but such
a modification for us would require a decision about how to weight
derivative conditions relative to others.

Another variant is the use of {\em non-interpolatory
approximations}.  Here, instead of using the same parameters
$\{w_j^{}\}$ in both the numerator and the denominator of (\ref{baryagain}),
which forces $n(z)/d(z)$ to interpolate $f$ at the support points,
we introduce separate parameters in $n(z)$ and $d(z)$ --- say,
$\{\alpha_j^{}\}$ and $\{\beta_j^{}\}$.  This is a straightforward
extension of the least-squares problem (\ref{l2prob}), with each
column of the Loewner matrix (\ref{nla}) splitting into two,
and the result is a rational approximation in barycentric form
that does not necessarily interpolate the data at any of the
sample points.  In such a case (which no longer has much link
to Antoulas and Anderson) we are coming closer to the method of
fundamental solutions as mentioned in the introduction, though
still with the use of the rational barycentric representation
(\ref{baryagain}).  The sample points and support points can be
chosen completely independently.  Our experiments suggest that
for many applications, non-interpolatory approximants are not so
different from interpolatory ones, and they have the disadvantage
that computing the SVD of the expanded Loewner matrix takes
about four times as long, which can be an issue for problems like
Application~\ref{sec-applics}.4 where $m$ is large.  (A related
issue of speed is the motivation of~\cite{berrut00}.)  However,
we now mention at least one context in which the non-interpolatory
approach may have significant advantages, enabling one to 
work with rational functions that satisfy an optimality
condition.

This is the matter of {\em iterative reweighting,} the basis of
the algorithm for minimax approximation known as {\em Lawson's
algorithm}~\cite{ew,lawson}.  Suppose an approximation $r\approx
f$ is computed by $m$ steps of the AAA algorithm.  Might it then
be improved to reduce the maximum error to closer to its minimax
value?  The Lawson idea is to do such a thing by an iteration
in which, at each step, least squares weights are adjusted in a
manner dependent on the corresponding error values at those points.
If the AAA calculation is run in the interpolatory mode, then
there is no chance of a true minimax approximation, since half the
parameters are used up in enforcing interpolation at prescribed
points, so for Lawson iterative reweighting, non-interpolatory
approximation is better.  Subsequent to the submission of the
original version of this paper, we have investigated this idea
and it is one of the key ingredients in the new {\tt minimax}
code mentioned in \S\ref{sec-app7}.  Details are reported in~\cite{fntb}.

Finally,
we have presented everything in the context of scalar functions
$f$ and $r$ of a scalar variable $z$, but in many applications $f$ may
be a vector or a matrix and $z$ may be a matrix.
So a natural extension of our work
is to {\em vector and matrix approximation.}

\section{\label{comparisons}Comparisons with
vector fitting and other algorithms}

Rational approximation is an old subject, and many algorithms have
been proposed over the years.  Probably nobody has a complete view of this
terrain, as the developments have occurred in diverse fields including
often highly theoretical approximation theory~\cite{aptek1,aptek,braess,
gaier,levinsaff,newman,pomm,sl,stahlabs,stahl,stahlspurious,vrc,walsh},
physics~\cite{bgm,froissart,gk03,gp97,gp99},
systems and control~\cite{antoulas,aa,belevitch,chdo,dmdd,dgb,
gab,gust,gs,ionita,dsp,sk},
extrapolation of sequences and series~\cite{bz,bz13,bz15},
and numerical linear algebra~\cite{bg1,gk,in}.
The languages and emphases differ widely among the fields, with the systems
and control literature, for example, giving particular attention to vector
and matrix approximation problems and to approximation on the imaginary axis.
Our own background is that of scalar approximation theory and
numerical analysis.  In this section we
offer some comments and comparisons of AAA with other algorithms,
especially {\em vector fitting} and {\em RKFIT}.

As mentioned at the outset, the two key features of the
AAA algorithm are (1) barycentric representation and (2)
adaptive selection of support points.  Both are essential if
one is to avoid exponential ill-conditioning.  Barycentric
representations of rational functions originated in the
numerical analysis community with Salzer~\cite{salzer} and Schneider
and Werner~\cite{sw} and were developed especially
by Berrut and his students and colleagues, including
Floater and Hormann~\cite{berrut,berrut00,bbm,bt,fh,klein}.
The work by Antoulas and Anderson was a largely independent
development~\cite{antoulas,aa,ionita}.  None of the works just
mentioned selects support points adaptively, however, and none
led to a widely used tool for numerical rational approximation.

The impact has been much greater of the related method known as {\em
vector fitting,} due to Gustavsen and Semlyen~\cite{gust,gs,hd} and
with links to the earlier ``SK'' algorithm of Sanathanan and Koerner~\cite{sk}.
Vector fitting produces a rational approximation represented
in a manner lying between polynomial quotients $r = p/q$ and
barycentric quotients $r = n/d$, namely a {\em partial fraction,}
$r = n$, essentially the numerator (\ref{baryagain}) without
the denominator.  In such a representation the numbers $z_k^{}$
are not free parameters: they must be the poles of $r$.  In vector
fitting, one begins with a decision as to how many poles will be
used and an initial estimate of their locations, which is then
adjusted iteratively by a process that makes
use of barycentric quotients, the aim being to minimize a nonlinear
least-squares norm.  As the iteration proceeds, the barycentric
support points converge to the poles of $r$, with the consequence
that when the iteration is complete, the denominator of the barycentric
quotient becomes~$1$ and 
one is left with the partial fraction approximation $r = n$.
This idea has been very successful, and vector fitting is widely
known in systems and control.

Comparing vector fitting with AAA, we note these differences: (i)
a partial fraction result instead of a barycentric quotient, (ii)
number of poles fixed in advance, (iii) estimates of poles needed
to start the algorithm, (iv) iterative process of pole adjustment
rather than support point adjustment,
(v) lack of flexibility in choosing support points adpatively,
since they are the same as the poles, (vi) easy imposition of
extra conditions such as poles in the left half-plane, (vii)
explicit norm to be minimized, and (viii)
extensive computational experience in applications involving
vectors as well as scalars.  Our view, based on limited experience,
is that although (vi)--(viii) may give vector fitting advantages in certain
applications, (i)--(v) make AAA, on the whole,
more likely to converge, more accurate, and much
easier to use.  In a moment we shall give some numerical evidence.

First, we wish to mention the other most competitive method we are
aware of: RKFIT (rational Krylov fitting), by Berljafa
and G\"uttel~\cite{bg1,bg}, which also
aims to minimize a least-squares norm.  RKFIT is connected also
to the IRKA (iterative rational Krylov) algorithm of
Gugercin, Antoulas, and Beattie~\cite{gab}.
In these methods the representation of $r$ is not
$p/q$, nor $n/d$, nor $n$, but a different form involving
{\em orthogonal rational functions}.  Again there is an iterative
process involving adjustment of poles, which in our experience is
less sensitive to a successful initial choice
of poles than with vector fitting.  Our view, based on limited
experience, is that RKFIT is often more accurate than vector
fitting, and more flexible as regards initial pole locations
(which may be put at $\infty$).  On the other hand, it seems
to be slower.  AAA seems to be typically both fast and
accurate, as well as having the big advantage of running in a black
box fashion with no user input of pole number or locations.

In Section~\ref{sec-app9} we applied AAA to the clamped beam
model of Chahlaoui and Van Dooren~\cite{chdo}.  It was our intention
to compare this performance against that of vector fitting, but our
experiments with vector fitting failed to give satisfactory convergence.
This appears to be a matter of initial guesses of pole locations, and
Drma\v c, et al.\ have had some success with this application
with their quadrature-based vector fitting method~\cite{dgb}.
This is among the experiments we have conducted that have shown
us that successful vector fitting can sometimes be challenging.

Aside from the matter of initial user inputs, the difficulty
we see in vector fitting is that it aims for a partial fraction
rather than barycentric representation of $f$, and this may be
ill-conditioned because it depends on the locations of poles,
which need not be favorably distributed.
(This effect is discussed in Section 4.3 of~\cite{dgb}.)
Figure~\ref{sec11fig} illustrates the differences
between monomial, partial fraction, and AAA barycentric bases
for Application~\ref{sec-applics}.5, involving approximation of
$1/J_0^{}$ on 2000 points in a rectangle in the complex plane.
If we worked with a $p/q$ representation
in the monomial basis, the associated $2000\times m$
Vandermonde matrix at step $m$, even after column rescaling, would
have condition numbers growing exponentially with $m$, and this is
shown in the highest curve of the figure.  Perhaps more surprising
is the exponential growth of condition numbers shown by the
next curve, labelled ``partial fractions.''   Here each column
of the $2000\times m$ Cauchy matrix (\ref{Cmatrix}) at step $m$
is obtained by sampling the function $1/(z-z_j^{})$,
where $z_j^{}$ is one of the poles of the associated AAA
approximant at step $m$.  (These constitute a good distribution
of poles for approximation of the given type, as might have been
obtained by a successful vector fitting iteration.)  Strikingly
better conditioned are the AAA Cauchy matrices, corresponding
to quotients $1/(z-z_j^{})$ in which $z_j^{}$ is a AAA support
point rather than a pole.
Evidently the support points distribute themselves in a manner that leads to
good conditioning.  While there
is great variation between problems, our experiments suggest that
the differences revealed in Figure~\ref{sec11fig} are typical.

\begin{figure}
\begin{center}
\vskip .2in
\includegraphics[scale=.6]{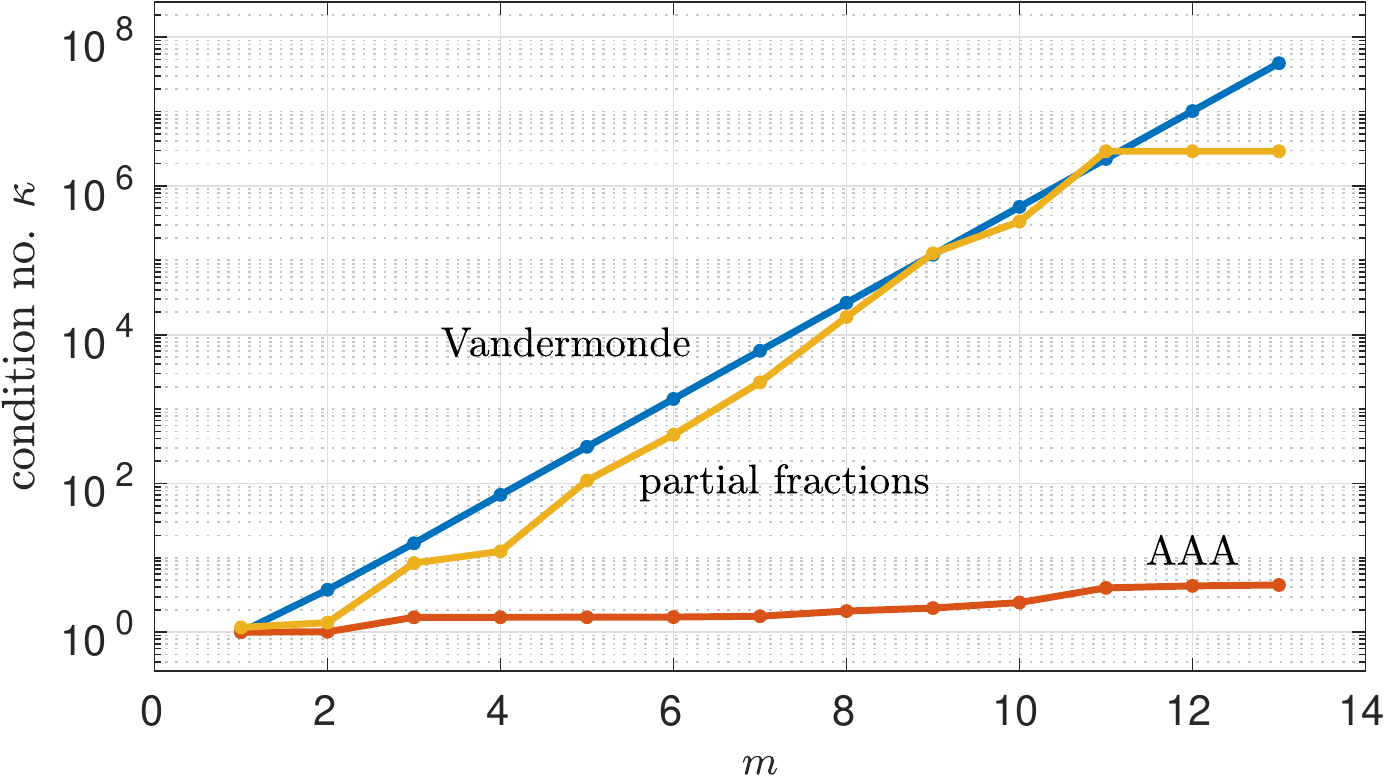}
\end{center}
\caption{\label{sec11fig}Condition numbers of bases
associated with $p/q$ quotients, vector
fitting, and the AAA algorithm for
Application~$\ref{sec-applics}.5$.  See the text for details.}
\end{figure}

Looking over the examples of this article, we note that
all but three of the Cauchy matrices $C$ (\ref{Cmatrix})
they utilize have condition numbers
close to $1$, and not greater than $40$; the exceptions are
the Cauchy matrices for the $\beta =256$ approximations of
Applications~\ref{sec-applics}.3 and~\ref{sec-applics}.4, with
$\kappa(C)\approx 10^4$, and for Application~\ref{sec-applics}.8,
with its exponentially graded grid, with $\kappa(C)\approx 10^8$.
Such anomalies are associated with support points that come very
close together.

\section{Conclusion}
The barycentric representation of rational functions is a flexible
and robust tool, enabling all kinds of computations that would be
impossible in floating point arithmetic based on an $r=p/q$
representation.  The AAA algorithm exploits the flexibility by
choosing support points adaptively.  The last section showed
that this leads to exceptionally well-conditioned bases for
representing numerators and denominators.

We do not imagine that the AAA algorithm is the last word in
rational approximation.  We do think, however, that its success shows
that barycentric quotients with adaptively selected
support points offer a strikingly effective way to compute with
rational functions.

We finish with one more example.  The following Matlab sequence
evaluates the Riemann zeta function at 100 points on the complex
line segment from $4-40i$ to $4+40i$, constructs a AAA approximant
$r$ of type $(29,29)$, and evaluates its poles and zeros.
This requires a bit less than a second on our desktop computer:

{\small
\begin{verbatim}

zeta = @(z) sum(bsxfun(@power,(1e5:-1:1)',-z))
[r,pol,res,zer] = aaa(zeta,linspace(4-40i,4+40i))

\end{verbatim}
\par}

\noindent 
The approximation has a pole at
$ 1.0000000000041 - 0.0000000000066i$ with residue
$0.999999999931 - 0.0000000014i$ and a zero at
$  0.500000000027 +14.134725141718i$.
These results match numbers for
$\zeta(z)$ in each case in all digits but the last two.

\section*{Acknowledgements}
We thank Anthony Austin, Jean-Paul Berrut,
Silviu Filip, and Stefan G\"uttel for
advice at many stages of this project.
It was G\"uttel who pointed us to (\ref{elliptic}).
In our ongoing work that has sprung from this publication so far,
reported in~\cite{fntb}, Filip has been a central figure.

\end{document}